\documentclass[12pt,reqno]{amsart}

\usepackage[margin=1in]{geometry}
\usepackage{psfrag}
\usepackage{graphicx}
\usepackage{epsfig} 
\usepackage{float}

\usepackage{amsmath,bm}
\usepackage{amsthm}
\usepackage{amsfonts}
\usepackage{amssymb}
\usepackage{verbatim}
\usepackage{latexsym}
\usepackage[mathscr]{eucal}
\usepackage{color}
\usepackage{enumerate}
\usepackage[active]{srcltx}

\newcounter{mnote}
\usepackage[colorlinks=true]{hyperref}
\hypersetup{urlcolor=blue, citecolor=red}

\numberwithin{equation}{section}

\title[Data assimilation using local observables]{Data assimilation for the Navier-Stokes equations using local observables}
\author{Animikh Biswas}
\author{Zachary Bradshaw}
\author{Michael S. Jolly}
\thanks{The research  of Z.~Bradshaw is supported in part by Simons Foundation grant 635438 and that of M.~Jolly in part by NSF grant DMS-1818754.}

\theoremstyle{plain}
\newtheorem{thm}{Theorem}[section]
\newtheorem{lem}{Lemma}[section]

\newtheorem{cor}{Corollary}[section]

\theoremstyle{definition}
\newtheorem{rmk}{Remark}[section]

\newcommand{\comments}[1]{}

\newcommand{\R}{\mathbb R}

\newcommand{\D}{\displaystyle }
\newcommand{\Z}{\mathbb Z}

\newcommand{\dt}{{\D\frac{d}{dt}}}

\newcommand{\N}{\mathbb N}
\newcommand{\be}{\begin{equation}}
\newcommand{\ee}{\end{equation}}
\newcommand{\bes}{\begin{equation*}}
\newcommand{\ees}{\end{equation*}}

\newcommand{\supp}{{\text{\rm supp }\,}}

\newcommand{\Om}{{\Omega_{0}}}
\newcommand{\om}{{\Omega}}

\newcommand{\ga}{\gamma}

\newcommand{\bu}{{\mathbf u}}
\newcommand{\bv}{{{\mathbf v}_N}}
\newcommand{\bvv}{{\mathbf v}}
\newcommand{\bw}{{\mathbf w}}
\newcommand{\bk}{{\mathbf k}}
\newcommand{\bx}{{\mathbf x}}
\newcommand{\bbf}{{\mathbf f}}
\newcommand{\by}{{\mathbf y}}

\newcommand{\I}{\infty}
\newcommand{\td}{\tilde}

\newcommand{\f}{\mathbf{f}}

\newcommand{\interp}{\mathcal{K}}

\newcommand{\mP}{\mathcal{P}}

\newcommand{\bea}{\begin{eqnarray}}
\newcommand{\eea}{\end{eqnarray}}
\newcommand{\beas}{\begin{eqnarray*}}
\newcommand{\eeas}{\end{eqnarray*}}

\newcommand{\EQ}[1]{\begin{equation}\begin{split} #1 \end{split}\end{equation}}

\newcommand{\charfn}[1]{{\raisebox{1.2pt}{\mbox{$\chi
_{\kern-1pt\lower3pt\hbox{{$\scriptstyle{#1}$}}}$}}}}

\date{\today}

\begin{document}
\begin{abstract}
We develop, analyze, and test an approximate, {\it global} data assimilation/synchronization algorithm based on purely {\it local} observations for the two-dimensional Navier-Stokes equations on the torus. We prove that, for any error threshold, if the reference flow is analytic with sufficiently large analyticity radius, then it can be recovered within that threshold.  Numerical computations are included to demonstrate the effectiveness of this approach, as well as variants with data on moving subdomains. In particular, we demonstrate numerically that machine precision synchronization is achieved for {\it mobile data} collected from a small fraction of the domain.
\end{abstract}

\maketitle

\section{Introduction} \label{sec.intro}

For a given dynamical system, which is believed to accurately describe some aspect(s) of an underlying physical reality, 
the problem of forecasting is often hindered by  inadequate knowledge of the initial state and/or model parameters describing the system.
However, in many cases, such as in weather prediction, this is compensated by the fact that one has access to data from (frequently noisy) measurements of the system, collected either continuously  in time or at discrete time points, albeit on a {\it much coarser spatial grid than the desired resolution of the forecast}.  \comments{This, for instance, is the case in atmospheric sciences where, since the launch of the first weather satellites in the 1960s, weather data has been collected nearly continuously in time, which furnishes us with the knowledge of the state of the system, e.g., the velocity vector field or temperature, on a {\it coarse spatial grid of points}.} The objective of data assimilation and signal synchronization in geophysics is to use low spatial resolution observational measurements
to fine tune  our knowledge of the state and/or model to improve the accuracy of the forecasts \cite{Daley1991, K}.
While atmospheric science, geoscience and meteorology have provided the initial impetus for the subject, it has now found widespread application, including, but not limited to, environmental sciences, systems biology and medicine \cite{Kost2, Kost-1}, imaging science, traffic control and urban planning, economics and finance and oil exploration \cite{ABN}. 

An extensive literature exists on data assimilation using the Bayesian and variational framework (e.g., 3D Var and Kalman filter based approaches)   \cite{ABN,BLSZ2013, BrML, BrGM, BrM, HM,K,KLS,LSZ,RC}. Despite this large body of work, 
as noted in \cite{HM1, tmk2016-1, tmk2016-2}, the
problems of stability, accuracy and {\it catastrophic filter divergence}, particularly for infinite-dimensional chaotic dynamical systems governed by PDEs, continue to pose serious challenges to rigorous analysis of these Bayesian/Kalman filter based schemes, and are far from being resolved.

Recently, Azouani, Olson and Titi developed a new approach to data assimilation using a feedback control paradigm via nudging/Newtonian relaxation, which is supported by rigorous analytical proof  of convergence of the scheme \cite{AOT, AT}. This initiated a lively field of research---see \cite{ANT,BOT,BCM,BM,CHL,FGHMMW,FMT,MTT} and the references therein.  More recently, it has been successfully implemented for the first time for efficient dynamical downscaling of a global circulation model \cite{desam2019}, where the authors assert that {\it ``overall results clearly suggest that CDA
provides an efficient new approach for dynamical downscaling by maintaining better balance between the global model and the downscaled fields."}

Rigorous results following in the vein of \cite{AOT}, or for instance \cite{BLSZ2013} for variational data assimilation,
are based on the earlier work on determining parameters \cite{FP,FT}
and {\it require global knowledge}, either in the form  of the low modes (which necessitates global measurements to determine) or information from a global array of uniformly distributed observables, usually nodal values or volume elements \cite{AOT}.  When measuring data from real world systems, certain considerations impact the placement of instruments. It is easier, for example, to measure fluid velocity or temperature at shallow depths in the ocean than at extreme depths. Similarly, a wide array of uniformly placed sensors are plainly infeasible when modeling the solar wind or the ionosphere, thus necessitating data assimilation techniques based on {\it local} measurements \cite{AK, Ion1, Ion2}. 
These examples illustrate the value of a \emph{global} data assimilation algorithm that is based on \emph{local} measurements. The existing abstract approaches to data assimilation require a global array of data points. In this paper, we introduce a data assimilation algorithm based on local observables for the 2D periodic Navier-Stokes equations that recovers a sufficiently regular reference solution within a specified error.  The key ingredient is a spectral inequality due to Egidi and Veseli\'c \cite{EV,EV2} which bounds  the $L^2$ norm over the full domain in terms of that over a subdomain, enabling us to use the local data obtained from the subdomain for global assimilation of the system. Moreover, we illustrate the efficacy of our algorithm by extensive computational studies. 
Furthermore, we demonstrate numerically that observation on a small fraction of the domain suffices for global assimilation provided the data collection domain moves with time to cover the entire domain, which may in practice indeed be the case for data collected by satellites, air crafts or ships, which are mobile. The efficacy of {\it mobile data}, i.e., moving point measurements for data assimilation was also computationally observed recently by Larios and Victor  \cite{LV} for a one-dimsnsional model. Our numerical work shows that the same holds true for data assimilation based on local observations for the Navier-Stokes equations.

The main results are stated in Section \ref{sec.main}.  After some preliminary material in Section \ref{sec.prelim}, we establish existence of solutions to the nudged equation in Section \ref{sec.existence}, followed by the proof of synchronization in Section \ref{sec.DA}. We provide in Section \ref{sec.comp} computational evidence to demonstrate the effectiveness of local sampling for the synchronization of global spatial flow features.  We then test in Section \ref{sec.mobile} variations of the algorithm in which the subdomain moves with time and find the convergence to the reference solution is greatly enhanced.     \comments{We note that Larios and Victor have done a computational study using mobile data for the 1D Allen-Cahn reaction-diffusion equation \cite{LV}.}  A brief summary is given in Section \ref{sec.summary}. It should be noted that while our local data assimilation results are established in the context of the Navier-Stokes equations,  it is not difficult to extend these results to other dissipative fluid models such as the magnetohydrodynamic equations or the Boussinesq equations.  

\section{Statement of main results}\label{sec.main}
 
Let $\Om$ be a $C^2$ domain in $\R^2$ or the periodic box $[-L/2,L/2]^2$. The Navier-Stokes equations in functional form   are 
\be \label{eq.nse}
\dt \bu + \nu A\bu +B(\bu,\bu)=\f \mbox{ in } \Om\times (0,\I);\quad \nabla \cdot \bu=0\mbox{ in } \Om\times (0,\I).
\ee
In \eqref{eq.nse}, $B(\bu,{\mathbf v})= \mP_\sigma( \bu\cdot\nabla {\mathbf v})$ where $\mP_\sigma$ is the Leray projection operator, $A$ is the Stokes operator, $\bbf$ is a given divergence free forcing and the velocity field $\bu$ is unknown.
We further impose zero Dirichlet boundary conditions if working on a $C^2$ domain and periodic boundary conditions if working on the torus. Additionally  assume the flow evolves from an initial datum $\bu_0 $ in an appropriate function space.  We use standard notations for the function spaces commonly used to analyze \eqref{eq.nse} \cite{cf,Temam}.  

Throughout this paper we consider a domain $\Om$ and a sub-domain $\om$. We partition $\om$ in the following way:~Form a lattice of points in $\om$ so that the distance between neighboring points is $h$. This leads to a finite collection of $M$ closed squares $\{S_i\}_{i=1}^M$  so that $\om\subset \cup_{1\leq i\leq M}S_i$ and $S_i\cap \om \neq \emptyset$ for all $1\leq i\leq M$. Let $\bx_i$ denote the center of each square.

For our first result, let $\Om=[-L/2,L/2]^2$ and consider \eqref{eq.nse} with periodic boundary conditions and analytic forcing $\bbf$.
Our goal is to obtain a global data assimilation result using local observables. In particular, observations are limited to an open set $\om$  compactly contained in $\Om$.   

Two types of interpolant operators appear in the literature. Type 1 operators satisfy an approximation inequality where the upper bound involves the $H^1$-norm. A relevant example  of a Type 1 operator is the following based on averages over volume elements: 
\EQ{
(I_{h}f)(x) = \sum_{i=1}^M \chi_{S_i}(\bx) (f)_{S_i},
}
where  $( f)_{S_i}$ denotes the integral average of $f$ over $S_i$, i.e., 
$(f)_{S_i}=\frac{1}{|S_i|}\int_{S_i} f$. 
The approximation property given in \cite{AOT} for this operator is the following:\EQ{\label{ineq.type1}
\|  f-I_hf  \|^2_{L^2(\Om)} \leq c_0 h^2 \|f\|^2_{H^1(\Om)}.
}
In our analysis we use a local version of this operator with an additional feature, {\it spectral filtering} to the first $N$ modes.  Our spectrally filtered operator is given by 
\EQ{
(I_{h,N,\om} f)(\bx)=P_N  \sum_{i=1}^M \chi_{S_i \cap \om}(\bx) (   f \chi_\om )_{S_i}\;,
}
where $P_N$ is the projector onto the first $N$ eigenvectors of $A$.
Denoting
\EQ{ \label{localinterp}
(I_{h, \om} f)(\bx)=  \sum_{i=1}^M \chi_{S_i \cap \om}(\bx) (   f \chi_\om )_{S_i},
}  
we have $(I_{h,N,\om} f)(\bx)=P_N(I_{h, \om} f)(\bx)$.
Note that $\supp I_{h, \om} f = \om$.
An identical (modulo constants) approximation property to \eqref{ineq.type1} will be proved in Section \ref{sec.prelim} for this local operator.

Type 2 operators  satisfy an approximation inequality where the upper bound involves the $H^2$-norm. A relevant example  of a Type 2 operator is the following based on nodal values: 
\EQ{
(\mathcal I_{h}f)(\bx) = \sum_{i=1}^M \chi_{S_i}(\bx) f(\bx_i).
}
The corresponding approximation property from \cite{AOT} is 
\EQ{\label{ineq.type2}
\|f - \mathcal I_hf\|^2_{L^2(\Om)}\leq c_0^2h^4 \|f\|^2_{H^2(\Om)}.
}
Again, we will investigate a local version of this operator. In particular, we define
\EQ{
(\mathcal I_{h,N,\om} f)(\bx)= P_N  \sum_{i=1}^N \chi_{S_i\cap \om}(\bx)  (\chi_\om  f )(\bx_i) =P_N   \mathcal I_h(\chi_\om  f).
}
Also let
\EQ{ \label{localnodal}
(\mathcal I_{h,\om} f)(\bx)=   \sum_{i=1}^N \chi_{S_i\cap \om}(\bx)  (\chi_\om f )(\bx_i).
}
An analogous inequality to \eqref{ineq.type2} will be proven in Section \ref{sec.prelim}.
Spectrally filtered operators using global observables similar to these have been used for data assimilation in \cite{COT}. 

The interpolant operators are used to feed information about a solution $\bu$ to \eqref{eq.nse} into the data assimilation equation 
\EQ{\label{eq.da1}
\frac d {dt} \bv +\nu A \bv +  P_NB(\bv,\bv) = P_N \bbf -\mu ( J_{h,N,\om}\bv -  J_{h,N,\om}\bu  ),
} 
where $J_{h,N,\om}\in \{ I_{h,N,\om} ,\mathcal I_{h,N,\om}\}$.
Note that the samples used to drive $\bv$ are confined to the sub-domain $\om$. Furthermore $\bv$ lives in $\operatorname{span}(\phi_1,\ldots,\phi_N)$. Our main results says that, within any given tolerance $\epsilon$, $\bv$ captures the long time properties of $\bu$ provided $N$  and $\mu$ are sufficiently large and $h$ is sufficiently small, with quantities determined by $\om$, $\nu$, the Grashof number $G$ (defined in \eqref{eq.Grashof}) and $\epsilon$. It requires the solution $\bu$, and consequently the forcing ${\mathbf f}$, to be  uniformly in the $L^2$ based Gevrey class, i.e.~$\bu \in L^\I ((0,\I);D(A^{1/2}e^{\sigma A^{1/2}}))$, with $\sigma$ sufficiently large as determined by $\om$, $\nu$, $G$ and $\epsilon$. The precise definition of $D(A^{1/2}e^{\sigma A^{1/2}})$ is given in Section \ref{sec.prelim}.

\begin{thm}[Approximate convergence for local observations] Let $\om$ be an open set in $[-L/2,L/2]^2$.
\label{thrm.main1} Let $\bu$ be the   solution to \eqref{eq.nse} for some $\bu_0\in V$ and $\bbf\in L^\I ((0,\I); H))$. Assume additionally that
\be \label{condn:gev}
M:=\limsup_{t\to \I} \| \bu(t)\|_{D(A^{1/2}e^{\sigma A^{1/2}}  )}=\limsup_{t\to \I}
|A^{1/2}e^{\sigma A^{1/2}}\bu|<\I.
\ee
Let $\epsilon>0$ be given. There exists a spectral index $N_*=N_*(\om,G,\nu,\epsilon,M)\in\N$, so that, for any $N\geq N_*$, there exists a large value $\mu(\nu,\om,G,N)$, a small value $h_*=h_*(\nu,\om,G,N)$ and a large length scale $\sigma_*=\sigma_*(\om)$ so that if $h\in (0,h_*)$ and $\sigma >\sigma_*$, then
\[
\| \bu(t) -\bv(t) \|_{L^2(\Om)}<\epsilon,
\]
for sufficiently large $t$, where $\bv$ is the global smooth solution to \eqref{eq.da1} taken with either $J_{h,N,\om}=  I_{h,N,\om}$ or $J_{h,N,\om}=  \mathcal I_{h,N,\om}$ and zero initial data.
\end{thm}
 
{\bf Comments on Theorem \ref{thrm.main1}:}
\begin{enumerate} 
\item The assumptions are complicated so we elaborate on how they fit together:
\begin{itemize}
\item The needed analyticity radius $\sigma$ is determined solely by the sub-domain. Therefore Theorem \ref{thrm.main1} can be viewed as saying: For analytic flows with analyticity radius large as determined by the sub-domain $\om$, we can approximately recover the solution. In this sense Theorem \ref{thrm.main1} is a conditional result.
\item Next, $N_*$ is chosen large in a manner depending on  $\epsilon$, $M$ and a priori bounded quantities associated with $\bu$. For $N\geq N_*$, we can execute the data assimilation argument for sufficiently large $\mu$ and sufficiently small $h$. For technical convenience, we fix $\mu$ in Theorem \ref{thrm.main1} based on $N$.  
\end{itemize} 
\item The existence of $\bv$ is implicit in Theorem \ref{thrm.main1}. The details are worked out in Section \ref{sec.existence}.
\item  If the solution $\bu$ is replaced by a Galerkin approximation in $\operatorname{span}(\phi_1,\ldots,\phi_N)$, then the conclusion of Theorem \ref{thrm.main1} holds whenever $\bbf \in L^\I(0,\I;H)$. Alternatively, if a condition is imposed on the forcing that results in the flow remaining spectrally local, e.g.~if the forcing and data are supported on finitely many Fourier modes, then the flow can be recovered exactly as $t \to \infty$.   This implies numerically simulated flows---which are spectrally localized---can be recovered up to machine precision and round-off error using observations confined to a sub-domain provided sufficiently many observations are taken. The number of observations depends exponentially on $N$, although as demonstrated numerically in Section \ref{sec.comp}, a far smaller value of $N$ is necessary. While this is interesting from a numeric perspective, in practical applications the observed data comes from nature and is not spectrally localized. When the flow is not spectrally local we lose exactness and also require analyticity.
\item In  Azouani-Olson-Titi data assimilation, the large parameter $\mu$ is attached to a large positive global quantity and is used to hide large contributions in the energy inequality for the difference of the solution to \eqref{eq.nse} and the solution to the data assimilation equation. In our localized setup, $\mu$ is attached to a large positive \emph{local} quantity while the contributions in the energy inequality remain global. To bridge the gap between global and local in this setting we use spectral inequalities developed to study   control problems. These are compiled in Section \ref{sec.prelim}
\end{enumerate}

An exact convergence result follows as a corollary of Theorem \ref{thrm.main1} provided \eqref{condn:gev} holds and one has full knowledge of $\bu|_{\om}$.   This means we solve
\EQ{\label{eq.dacomplete}
\frac d {dt} \bv +\nu A \bv +  P_NB(\bv,\bv) = P_N \bbf -\mu P_N \chi_\om(\bv - \bu  )\;.
} 
In particular, we can construct a vector field $\bvv$ that converges to $\bu$ as $t\to \I$ in an appropriate average sense by increasing the sample size in Theorem \ref{thrm.main1}. 
\begin{cor}\label{thrm.exact.convergence.1}
Under the assumptions of Theorem \ref{thrm.main1}, i.e. \eqref{condn:gev}, there exists a vector field $\bvv\in L^\I(0,\I;H) \cap L^2(0,T;V)$ for all $T>0$ so that  $\bvv$ is a limit (in an appropriate sense) of a sequence of vector fields satisfying \eqref{eq.dacomplete} and for every measurable set $U$ we have 
\begin{itemize}
\item  for every $\Delta>0$,
\[
\lim_{t\to \I} \int_{t}^{t+\Delta} \int_U (\bu-\bvv)\,d\bx\,dt = 0,
\]
at an exponential rate;
\item there exists a set of times $S$ with $|S^c|=0$ so that
\[
\lim_{t\to \I}\chi_S(t) \int_{U} (\bu-\bvv)(\bx,t) \,d\bx = 0,
\] 
at an exponential rate. 
\end{itemize}
\end{cor}
This asserts only that a vector field matching the long time behavior of $\bu$ exists given complete local knowledge of $\bu$. The fact that it is obtained as a limit of the solution to \eqref{eq.dacomplete} means it can be approximated numerically. However, this corollary does not say what system governs $\bvv$. This is an interesting direction for future research.

\bigskip We also obtain an exact data assimilation scheme for rougher forcing provided the sub-domain is sufficiently large within the domain. For convenience we work with a bounded domain $\Om$ with smooth boundary and impose zero Dirichlet boundary conditions. We consider a variant of the local data assimilation equation, namely,
\EQ{\label{eq.da2}
\frac d {dt} \bvv +\nu A \bvv +   B(\bvv,\bvv) =   \bbf -\mu I_{h,\om}(\bvv -  \bu).
} 
Note that the interpolant operators  are localized but there is no spectral projection. We manage this by requiring the sub-domain occupy almost the full domain as this allows us to use a helpful observability inequality \cite{XiYo}. 
\begin{thm}[Exact convergence for large subdomains]\label{thrm.main2} Let $\om\subset \Om$ be an open set with smooth boundary. 
Let $\bu$ be the  solution to \eqref{eq.nse} for some $\bu_0\in V$ and $\bbf\in L^\I ((0,\I);H)$. Let $\epsilon>0$ be given. Then there exists a large value $\mu_*=\mu_*(\nu,\om,\bbf)$ and a small value $h_*=h_*(\nu,\om,\mu)$   so that for  any $\mu>\mu_*$ and $h\in (0,h_*)$,   we have
\[
\| \bu(t) -\bvv(t) \|_{L^2(\Om)}\to 0, \qquad \text{as}  \quad t \to \infty
\]
at an exponential rate
provided $d_H(\partial\om ,\partial \Om) \sim (\sqrt \lambda_1 G)^{-1}$, where $\bvv$ is the solution to \eqref{eq.da2} taken with $J_{h,N,\om}=  I_{h,N,\om}$   and zero initial data and $d_H$ denotes the Hausdorff distance between two compact sets.
\end{thm}
As we will explain in Remark \ref{rmk.da2}, this result allows one to avoid the collection of measurements near the (possibly turbulent) boundary layer, which may be inherently error prone. Due to Remark \ref{rmk.da2}, $d_H(\partial\om ,\partial \Om)$ is comparable to the value $h$ found in \cite{AOT} meaning that the volume elements adjacent to the boundary of $\Om$ can be eliminated from the interpolant operators in \cite{AOT}.  Although this may be a  small number of volume elements, they are adjacent to the boundary layer, which may be turbulent for flows with large Reynolds numbers and, consequently, subject to large measurement error. 
It is interesting also from a mathematical viewpoint as well, since it is an exact convergence result based on a local interpolant.  As the argument is similar to Theorem \ref{thrm.main1}, we only sketch the details of a proof---see Remarks \ref{rmk.existence.da2} and \ref{rmk.da2}

We note that the localization problem for determining nodes in the sense of \cite{JonesTiti2} has been solved for analytic forcing \cite{FR,FKR}. A general theme is that data assimilation implies determining quantity type results but not the other away around. For the localization problem, this appears to be the case. Indeed, the argument in \cite{FR,FKR} is applied at the level of elements of the attractor, and both solutions are analytic. In data assimilation, the solution to the localized data assimilation equation does not {\it a priori} converge to the global attractor for $\bbf$. Furthermore, $\bv$ is not analytic because it is driven by a term with compact support. Hence there are clear barriers to adapting the methods in \cite{FR,FKR} to the data assimilation problem.

\section{Preliminaries}\label{sec.prelim}

\subsection{Strong solutions to the 2D Navier-Stokes}
Recall that given $\bu_0\in V$ and $\bbf\in L^\I((0,\I);H)$, \eqref{eq.nse} has a unique global solution $\bu$ so that 
\[
\bu\in C([0,T];V)\cap L^2((0,T);D(A)) \text{ and }\frac {d\bu} {dt}\in L^2((0,T);H),
\]
for every $T>0$ \cite{cf}. Furthermore we have that there exists a time $t_0=t_0(\bu_0)$ so that,
for all $t\geq t_0$,\[
\|\bu(t)\|_{L^2(\Om)}^2\leq 2 \nu^2G^2 \text{ and }\int_t^{t+T} \|A^{1/2}\bu(s)\|_{L^2(\Om)}^2\,ds \leq 2(1+T\nu\lambda_1)\nu G^2,
\] 
where $T>0$ is fixed and $G$ denotes the Grashof number which is defined to be 
\EQ{\label{eq.Grashof}
G= \frac 1 {\nu^2\lambda_1}\limsup_{t\to \I} \|\bbf(t)\|_{L^2(\Om)}.
}
The above are true for both Dirichlet boundary conditions on a bounded domain with $C^2$ boundary or for periodic boundary conditions. For the periodic case we have the following improvement: There exists a time $t_0=t_0(\bu_0)$ so that,
for all $t\geq t_0$,\[
\|A^{1/2}\bu(t)\|_{L^2(\Om)}^2\leq 2 \nu^2\lambda_1 G^2 \text{ and }\int_t^{t+T} \|A\bu(s)\|_{L^2(\Om)}^2\,ds \leq 2(1+T\nu\lambda_1)\nu\lambda_1 G^2,
\] 
where $T>0$ is fixed.

\subsection{$L^2$ Gevrey classes}\label{subsec.gevrey}
Recall that if $\bu\in L^2(\Om)$ is periodic and has zero mean  and $\Om= [-L/2,L/2]^2$, then 
\[
\bu (\bx)= 
\sum_{\bk\in \Z^2\setminus\{0\}} \hat \bu_\bk e^{\frac{2\pi i}L \bk\cdot \bx} ,
\]
where 
\[
 \hat \bu_\bk  = \int_\Om \bu(\by) e^{-\frac{2\pi i}L \bk\cdot \by}\,d\by.
\]

Working on the periodic box $[-L/2,L/2]^2$, we define the Gevrey space $D(e^{\sigma A^{s}})$ to be those elements of $H$ satisfying 
\[
\| \bu\|_{D(e^{\sigma A^{s}})}^2:=  L^2\sum_{k\in \Z^2} e^{2\sigma|2\pi \frac {k} L|^{2s}}|\hat \bu_\bk|^2<\I.
\]
Analyticity corresponds to $s=1/2$.
Note that for Gevrey class forcing, a solution $\bu$ to \eqref{eq.nse} becomes and remains Gevrey regular for positive times. Indeed, for large enough times we have the following uniform bound \cite[p.~74]{FMRT}
\[
|\hat \bu_\bk|^2\leq C \lambda_1 \nu^2 |\bk|^{-2} e^{-4\pi \delta_0|\bk| /L}[1+G^2]
\]
where $\delta_0$ is inversely related to $G$.  For our applications, this is insufficient so we must impose an explicit assumption that the analyticity radius is large. In particular, we assume 
\[
\limsup_{t\to \I} \| \bu(t)\|_{D(A^{1/2}e^{\sigma A^{1/2}}  )}=:M<\I,
\]
where 
\[
\| \bu(t)\|_{D(A^{1/2}e^{\sigma A^{1/2}}  )} = |A^{1/2}e^{\sigma A^{1/2}} \bu|.
\]
The preceding condition implies 
\[
\limsup_{t\to \I} |\hat \bu_\bk|^2(t) \leq  M^2  |\bk|^{-2} e^{-4\pi \sigma|\bk| /L}
\]
Gevrey class and analytic solutions to \eqref{eq.nse} and other fluid models have been studied extensively. A partial list is \cite{Biswas-Swanson,BGK,CKV,FoTe,GK1,K1}.

\subsection{The Stokes operator}\label{subsec.stokes}
We denote by $\phi_i$ the eigenvectors and $\lambda_i$ the eigenvalues of the Stokes operator.  $P_N$ denotes the projection operator from $L^2$ onto $\operatorname{span}(\phi_1,\ldots,\phi_N)$. 

Recall from \cite{cf} that for periodic domains and restricting to functions with zero mean that the Stokes operator $A$ agrees with $-\Delta$. Furthermore, abusing notation slightly, the eigenfunctions $\phi_\bk$ of $A$ can be written explicitly in terms of $\{ e^{2\pi i \frac \bk L \cdot \bx}\}_{\bk\in \Z^2}$ as
\[
\phi_\bk = a_\bk L^{-1}e^{2\pi i \frac \bk L \cdot \bx} +\bar a_{\bk}L^{-1}e^{-2\pi i \frac \bk L \cdot \bx},
\] 
where $a_\bk\in \mathbb C^2$ satisfy $a_\bk\cdot \bk =0$.  For each $\bk\in \Z^2\setminus \{0\}$ there are actually two eigenfunctions of the above form but we suppress this. Note that $\{\phi_\bk\}$ is orthonormal and all elements have mean zero. The eigenvalues of $A$ are the values $4\pi^2L^{-2}|\bk|^2$ for $\bk\in \Z^2\setminus 0$. The eigenfunctions can be ordered as $\{\phi_j\}_{j\in \N}$ so that the corresponding eigenvalues $\lambda_j$ are nondecreasing. By symmetry, the multiplicity of the eigenvalue $\lambda_j$ is 
\[
 \#\{\bk\in \Z^2:|\bk|^2= \lambda_j  \lambda_1^{-1} \}.
\]
If we are given an eigenvalue $\lambda_j$, then this corresponds to points $\bk\in \Z^2\setminus \{0\}$ in the square $[-K,K]^2$ where $ K^2\sim j$. Furthermore, we have asymptotically that 
\[
\lambda_j \lesssim j,
\]
implying 
\[
K^2\lesssim j.
\]
Plainly then, if $\bu \in \operatorname{span}(\phi_1,\ldots,\phi_N)$, there exists $K\sim \sqrt{N}$ so that $\hat \bu$ is  supported in $[-K,K]^2$.

For other properties of the Stokes operator, as well as its eigenvectors and eigenvalues, see \cite[II.6]{FMRT} as well as \cite{cf,Temam}.

\subsection{Approximation property of local interpolant operators}
We prove analogues of \eqref{ineq.type1} and \eqref{ineq.type2} for  local interpolant operators.
The local volume interpolant  operator \eqref{localinterp} satisfies the  approximation property
\EQ{  \label{localvolineq}
\|I_{h,\om}f -f \|_{L^2(\om)}^2\leq c_0 h^2 \|f\|_{H^1(\Om)}^2.
}
Structurally this is identical to \eqref{ineq.type1} but the operator is local so we check details.
By the Poincar\'e inequality,
\[
\|I_{h,\om}f -f \|_{L^2(\om)}^2 \leq \sum_{i=1}^M |S_i\cap \om| C_{S_i\cap \om} \|  \nabla (f\chi_\om)\|_{L^2(S_i\cap \om)}^2,
\]
where $ C_{S_i\cap \om}$ is the Poincar\'e constant for $S_i\cap \om$  and $|\cdot|$ denotes 2D Lebesgue measure. These constants are uniformly bounded because the sets $ C_{S_i\cap \om}$  all have bounded diameters. We thus obtain
\[
\|I_{h,\om}f -f \|_{L^2(\om)}^2 \leq c_0h^2 \sum_{i=1}^M   \|  \nabla f\|_{L^2(S_i\cap \om)}^2 = c_0h^2   \|f\|_{H^1(\om)}^2 \leq c_0h^2   \|f\|_{H^1(\Om)}^2,
\]
which is \eqref{localvolineq}.

The local nodal interpolant operator \eqref{localnodal}   satisfies the  approximation property
\be  \label{localnodalineq}
\|\mathcal I_{h,\om}f -f \|_{L^2(\om)}^2 \leq c_0^2 h^4 \|f\|_{H^2(\Om)}^2.
\ee
Again, this follows essentially  the original argument in \cite{AOT} which is adapted from \cite{JonesTiti2}. 
Recall from \cite{JonesTiti2,AOT} that if $Q$ is a cube of side length $h$ and $\bx,\by\in Q$, then
\EQ{
|\phi(\bx)-\phi(\by)|\leq c_0 h \|A\phi\|_{L^2(Q)}.
}
Then,
\EQ{
|\chi_\om f - \mathcal I_{h,\om} f |^2&= \sum_{i=1}^N \int_{S_i\cap \om}|f(\bx)-f(\bx_i)|^2\,d\bx
\\&\leq \sum_{i=1}^N \int_{S_i }|f(\bx)-f(\bx_i)|^2\,d\bx
\\&\leq \sum_{i=1}^N c_0^2 h^2 |S_i| \|Af\|_{L^2(S_i)}^2
\\&\leq c_0^2 h^4 \|f\|_{H^2(\Om)}^2.
}

\subsection{Spectral inequalities}

 For our approximate data assimilation result, we use a spectral inequality of Egidi and Veseli\'c for the torus \cite{EV,EV2}. This extends earlier work on $\R^d$ \cite{Ko}. The spectral inequality applies to ``thick'' sets. A set $S$  is thick in $\R^2$ if there exists $\ga\in (0,1]$ and $a= (a_1,a_2)$ where $a_i> 0$ so that for every $\bx\in \R^2$,
\[
 | (S+\bx)\cap ([0,a_1]\times [0,a_2])| \geq \ga a_1a_2\;.
\]
  It is easy to see that any open set in $[-L/2,L/2]^2$ which is periodically extended to $\R^2$ is thick.
The spectral theorem on the torus is the following.
\begin{thm}[\cite{EV}]\label{thrm.spectral.torus}
Let $f\in L^2(\Om)$ where $\Om$ denotes the torus $[0, L_1]\times [0, L_2]$. Assume $\supp \hat f \subset J$ where $J$ is a rectangle in $\R^2$ with sides parallel to coordinate axes and of length $b_1$ and $b_2$. Let $b=(b_1,b_2)$. Let $S\subset \R^2$ be a $(\ga,a)$-thick set with $a=(a_1,a_2)$ so that $0<a_j<2\pi L_j$ for $j=1,2$. Then 
\[
\| f\|_{L^2(\Om)}\leq C \gamma^{-c a\cdot b  - \frac {13}2} \|f\|_{L^2(S\cap \Om)}.
\]
\end{thm}
For simplicity we  take $L_1=L_2=L$ and $S$ to be the periodic extension of a ball with radius $r<L/2$ to all of $\R^2$. It is not difficult to see that this set is thick with 
$$
a_i = L-r\;, \qquad \text{and}\qquad \ga= \frac {\pi r^2}{4(L-r)^2}\;.
$$ 
We can also take $b_1=b_2=2K$ where $K\in \N$ is fixed and $J$ centered at the origin. Then, for any open set $\om$ contained in $\Om$, we have as a consequence of Theorem \ref{thrm.spectral.torus} applied to a ball of radius $r$ contained in $\Om$ that 
\EQ{\label{ineq.spectral.torus.1}
\sum_{\bk\in [-K,K]^2\cap \Z^2} | \td f_\bk|^2  \leq C \gamma^{-2c (L-r)K  -13} \int_\om \bigg| \sum_{\bk\in [-K,K]^2\cap \Z^2} \td f_\bk e^{2\pi i \frac k L \cdot \bx}  \bigg|^2\,d\bx.
}
where $\td f_\bk$ is the Fourier coefficient for $\bk\in \Z^2$. We prefer to re-write this inequality in the following form
\EQ{
\label{ineq.spectral.torus.2} \|f\|_{L^2(\Om)}^2 \leq  C_{\om} e^{ C_{\om} K} \|f\|_{L^2(\om)}^2,
}  
where  $C_\om$ represents positive constants which are independent of $K$ and $\operatorname{range}\hat f \subset [-K,K]^2$. Based on the discussion in Section \ref{subsec.stokes}, we can also formulate this result in terms of the Stokes operator: If $f\in \operatorname{span}(\phi_1,\ldots ,\phi_N)$, then
\EQ{
\label{ineq.spectral.torus.3} \|f\|_{L^2(\Om)}^2 \leq  C_{\om} e^{ C_{\om} \sqrt N} \|f\|_{L^2(\om)}^2.
}  

For bounded domains,  there is a spectral inequality for the Stokes operator due to Chaves, Silva and Lebeau \cite{CSL}.
\begin{thm}[\cite{CSL}]\label{thrm.CSL}
Let $\om \subset \Om$ be a nonempty open set. There exist constants $M>0$ and $K>0$ so that for every sequence of complex numbers $z_j$ and every real $\Lambda>0$ we have 
\EQ{\label{ineq.spectral}
\sum_{\lambda_j\leq \Lambda} |z_j|^2 =\int_{\Om} \bigg| \sum_{\lambda_j \leq \Lambda} z_j  \phi_j \bigg|^2\,dx \leq Me^{K\sqrt{\Lambda}} \int_{\om} \bigg|   \sum_{\lambda_j\leq \Lambda} z_j\phi_j (\bx)  \bigg|^2\,d\bx
}
where $\phi_j$ are the eigenvectors and $\lambda_j$ are the eigenvalues of the Stokes operator.
\end{thm} 

Spectral inequalities of this form were  established earlier for elliptic operators on a bounded domain in \cite{LebeauZuazua,RL2012} using Carleman inequalities and a pointwise interpolation estimate from \cite{LebeauLeRobbiano}. 
Although in Theorem \ref{thrm.main1}, we consider only the case of the periodic boundary conditions, similar techniques can be employed for the bounded domain as well ---see Remark \ref{rmk.bounded.domain}.

\subsection{An observation inequality}
Theorem \ref{thrm.main2} is an exact convergence result when the observation domain is almost the entire domain. Our main technical tool for this is the following observation inequality due to Xin and Yongyong \cite{XiYo}.
\begin{lem}[\cite{XiYo}]\label{lemma.xi.yo} Let $\om$ and $\Om$  be bounded domains with smooth boundary so that $\om\subset \Om$.   
For any $\epsilon>0$, there exists $K(\epsilon)>0$ so that for $k>K$, the following inequality holds
\EQ{\label{ineq.spectral.bounded}
\int_{\Om} |\nabla u|^2 + k \chi_\om |u|^2\,d\bx \geq (\lambda_1(\om) - \epsilon) \int_{\Om} |u|^2\,d\bx,
}
for $u\in H_0^1(\Om)$ and $\lambda_1$ the first eigenvalue of the Laplace operator on the domain $\Om\setminus \overline{\om}$ with zero-Dirichlet boundary conditions. As $|\om| \to |\Om|$
, $\lambda_1(\om)$ increases without bound. 
\end{lem}

As a final comment let us note that an observation inequality for \eqref{eq.nse} is given in \cite{IK} for the difference of two solutions. It cannot, however, be applied to the difference of the reference solution and the data assimilation solution.

\section{Solving the data assimilation equations}\label{sec.existence}

In this section we construct solutions to the data assimilation equations introduced in Section \ref{sec.intro}.  
Recall the data assimilation equation is 
\EQ{ 
\frac d {dt} \bv +\nu A \bv +  P_NB(\bv,\bv) = P_N \bbf -\mu  J_{h,N,\om}(\bv -  \bu  ),
}
where $J_{h,N,\om}$ is either $I_{h,N,\om}$ or $\mathcal I_{h,N,\om}$.
We focus on periodic boundary conditions.
Notice that the data assimilation equation has the form 
\[
\frac d {dt} \bv+\nu A \bv = F_N(\bu,\bv,\bbf),
\]
where $ F_N \in \operatorname{span}( \phi_1,\ldots,\phi_N)$. Provided the data $\bv_0$ is also in $\operatorname{span}( \phi_1,\ldots,\phi_N)$, we may seek a solution $\bv(t) \in \operatorname{span}( \phi_1,\ldots,\phi_N)$ for all $t$.  Formally taking $\bv(t) \in \operatorname{span}( \phi_1,\ldots,\phi_N)$ for all $t$ results in a finite system of ODEs which possesses a local-in-time unique strong solution.   We take $\bv$ to be this solution and let $T(\bv_0)$ be the maximal existence time for $\bv$ for data $\bv_0$.  Since the solution is strong, we have $\bv\in C([0,T(\bv_0));H)$.  Note that if $T(\bv_0)$ is maximal and finite, then the $L^2$ norm of $\bv$ must blow up at $T(\bv_0)$. Hence, a  uniform in time bound for the $L^2$ norm implies $T(\bv_0)=\I$. 
The next lemma provides such a bound and implies $\bv$ is a global solution provided $h$ is sufficiently small.

\begin{lem}\label{lem.da.exist}  Let $\bu_0\in V$ be given and let $\bu$ be the solution to \eqref{eq.nse} for data $\bu_0\in V$ and $\bbf\in L^\I(0,\I;H)$. Fix $N\in \N$, $h>0$ and $\mu>0$ and $J_{h,N,\om}\in \{I_{h,N,\om}, \mathcal I_{h,N,\om}  \}$. Assume $\bv_0\in  \operatorname{span}( \phi_1,\ldots,\phi_N)$ and let $\bv$ be the unique strong solution to \eqref{eq.da1} on $[0,T(\bv_0))$ for some $h>0$. Then, provided   $h $ is sufficiently small,  $\bv$ satisfies
\EQ{
\bv \in L^\I((0,T(\bv_0));H) \cap L^2((0,T');V),
}
for every $0<T'<T(\bv_0)$.
The first inclusion implies $T(\bv_0)=\I$.
\end{lem}

For type 1 interpolants, the requirement on $h$ in \cite{AOT} is $2\mu c_0  h^2 \leq \nu$. This is better than ours by a factor of $4$.   

Because $\bv $ is confined to $\operatorname{span}( \phi_1,\ldots,\phi_N)$, we may deduce bounds on higher order derivatives freely using properties of the eigenvectors of the Stokes operator. Hence we do not need to prove such estimates as required in \cite{AOT}.

We do not require any Gevrey regularity of $\bbf$ at this point. 

\begin{proof} We first focus on $J_{h,N,\om}=I_{h,N,\om}$. 
Take the inner product of \eqref{eq.da1} with $\bv$ and integrate in space to obtain
\[
\frac 1 2 \frac d {dt} \|\bv\|_{L^2(\Om)}^2 +\nu \|A^{1/2}\bv\|_{L^2(\Om)}^2 = (\bbf+\mu   I_{h,N,\om}\bu,\bv) - \mu (   I_{h,N,\om}\bv   ,\bv   )\;,
\]
where $(\cdot,\cdot)=(\cdot,\cdot)_{L^2(\Om)}$.
We estimate each term on the right hand side. For the source terms we have \[  |(\bbf,\bv) |\leq \frac {4 }{ \nu \lambda_1}\|\bbf\|_{L^2(\Om)}^2+\frac   {\nu} 4 \|A^{1/2}\bv\|_{L^2(\Om)}^2,\]
and
\[
|(\mu I_{h,N,\om}\bu,\bv) | \leq  \frac {4 \mu^2} {\nu \lambda_1} \| I_{h, \om}\bu\|_{L^2(\Om)}^2       +\frac  \nu 4 \|A^{1/2}\bv\|_{L^2(\Om)}^2,
\]
where we used the fact that $\bv$ is projected onto the first $N$ modes to eliminate $P_N$ in the inner product.
A direct computation confirms that 
\[
 \| I_{h, \om}\bu\|_{L^2(\Om)}^2 \leq \|\bu\|_{L^2(\Om)}^2.
\]
where we used \eqref{ineq.type1}.  Hence\[
|(\mu I_{h,N,\om}\bu,\bv) | \leq  \frac {4 \mu^2} {\nu \lambda_1} \|  \bu\|_{L^2(\Om)}^2       +\frac  \nu 4 \|A^{1/2}\bv\|_{L^2(\Om)}^2. 
\]

For the remaining term we have 
\EQ{\label{eq.6.18.20a}
- \mu (   I_{h,N,\om}\bv   ,\bv   ) = -\mu(I_{h,\om} \bv ,\bv) = -\mu(I_{h,\om} \bv- \chi_\om \bv,\bv) - \mu \int_\om |\bv|^2\,d\bx. 
}
The last term above has a good sign while, using \eqref{localvolineq},
the second to last is bounded as 
\EQ{\label{ineq.interp.bound.type1}
\mu|(I_{h,\om} \bv-  \chi_\om \bv,\bv)| & \leq \mu \|  I_{h,\om} \bv-  \chi_\om \bv\|_{L^2(\Om)}\|A^{1/2}\bv \|_{L^2(\om)} \\
 &\leq   2\mu c_0  h^2 \|A^{1/2}\bv \|_{L^2(\Om)}^2   +   \frac \mu {2} \|\bv\|_{L^2(\om)}^2.
}
We now require 
\be \label{mucond}
 2\mu c_0  h^2 \leq \frac \nu 4.
\ee
Granting this and absorbing terms where possible we obtain 
\EQ{
&\frac 1 2 \frac d {dt} \|\bv\|_{L^2(\Om)}^2 +\frac \nu 4 \|A^{1/2}\bv\|_{L^2(\Om)}^2 \leq \frac {4 }{ \nu \lambda_1}\|\bbf\|_{L^2(\Om)}^2 +   \frac {4 \mu^2} {\nu \lambda_1} \|  \bu\|_{L^2(\Om)}^2 -\frac \mu 2 \| \bv\|_{L^2(\om)}^2.
}
Using the Poincar\'e inequality and dropping the term with a good sign we have
\EQ{
&\frac 1 2 \frac d {dt} \|\bv\|_{L^2(\Om)}^2 +\frac {\nu \lambda_1} 4 \|\bv\|_{L^2(\Om)}^2 \leq \frac {4 }{ \nu \lambda_1}\|\bbf\|_{L^2(\Om)}^2 +   \frac {4 \mu^2} {\nu \lambda_1} \|  \bu\|_{L^2(\Om)}^2.
}
Since the right hand side is uniformly bounded in $t$, this leads to a uniform in time bound on $|\bv|$ in the usual way \cite{AOT}. Note that this bound is independent of $N$

The proof is the same if we replace $I_{h,N,\om}$ with $\mathcal I_{h,N,\om}$ with one modification: Instead of \eqref{ineq.interp.bound.type1} we have  
\EQ{
 \mu |(\chi_\om   \bv  -   \mathcal I_{h,\om }\bv    ,\bv   )| & \leq \mu \|    \chi_\om  \bv   -   \mathcal I_{h,\om}  \bv\|_{L^2(\Om)}\|\bv\|_{L^2(\Om)}
\\&\leq   2\mu  c_0^2 h^4  \| \bv \|_{H^2(\Om)}^2      +  \frac \mu {2}  \|\bv\|_{L^2(\om)}^2
\\&\leq 2 \mu  c_0^2 h^4    \lambda_N\|A^{1/2}\bv \|_{L^2(\Om)}^2   + \frac \mu {2} \|\bv\|_{L^2(\om)}^2.
}
After fixing $N$ and $\mu$ we therefore take $h$ small so that $2\mu c_0^2 h^4 \lambda_N <\frac \nu 4$ and proceed as in the case of $I_{h,N,\om}$.

\end{proof}
 
\begin{rmk} \label{rmk.existence.da2}
We discuss the existence problem for \eqref{eq.da2}. Because the localization in \eqref{eq.da2} does not involve a spectral projection, it is very similar to the usual existence result \cite{AOT}, using the usual Galerkin approximation procedure. Therefore we include only the  needed a priori bound and direct the reader to \cite{AOT} as well as \cite{cf,Temam} for more details. For the a priori bound,  starting with \eqref{eq.da2} we have 
\[
\frac 1 2 \frac d {dt} \|\bv\|_{L^2(\Om)}^2 +\nu \|A^{1/2}\bv\|_{L^2(\Om)}^2 =  ( \bbf ,\bv) -\mu  ( I_{h,\om}\bv,\bv) +  \mu(I_{h,\om}\bu,\bv).
\]
As in the proof of Proposition \ref{lem.da.exist}, the source terms lead to time-independent quantities on the right hand side while 
\[
-\mu  ( I_{h,\om}\bv,\bv)  = -\mu ( I_{h,\om}\bv -\chi_\om\bv,\bv) -\mu \|\bv\|_{L^2(\om)}^2.
\]
This is identical to \eqref{eq.6.18.20a}  and we can conclude following the identical argument. In particular, for $\mu$ fixed and $h$ chosen to satisfy  $8\mu c_0  h^2 \leq \nu$, we obtain a uniform bound on $\bv$ in terms of $\mu$, $\bu$ and $\bbf$. To rigorously construct a solution, we would now apply this a priori bound to a Galerkin scheme and pass to the limit using the standard compactness argument.

This existence result does not require $\om$ to occupy most of $\Om$. This constraint will be needed for data assimilation.
\end{rmk}

\section{Local data assimilation}\label{sec.DA}
%
%
%
%

\begin{proof}[Proof of Theorem \ref{thrm.main1}] Let $\epsilon>0$ be given.  Let $\bar \epsilon=\frac {\epsilon \nu \lambda_1} 8$.
Let $\bu$ and $\bv$ be as in the statement of Theorem \ref{thrm.main1}.
Note that for any $N\in \N$, $P_N\bu$ solves
\EQ{\label{eq.nse.projected}
&\frac d {dt} P_N\bu +\nu A P_N\bu + P_NB(P_N\bu , \bu)    = P_N\bbf - P_NB(Q_N\bu, \bu) ,
}
where $Q_N = I - P_N$. We will eventually specify a value for $N$.


 Let $\bw =  \bv -P_N \bu$.
Then $\bw$ is divergence free and satisfies
\EQ{\label{eq.w}
&\frac d {dt} \bw +\nu A \bw + P_N B(P_N \bu,\bw)+P_NB(\bw,\bw)+ P_NB(\bw,P_N\bu)
\\& = -\mu   I_{h,N,\om}\bw +\mu I_{h,N,\om}Q_N\bu +P_NB(Q_N\bu, \bu) +P_NB(P_N\bu,Q_N\bu) .
}
Throughout this proof, unless otherwise adorned, $\|\cdot\|=\|\cdot\|_{L^2(\Om)}$.  We have 
\EQ{\label{ineq.w.energy.1}
&\frac 1 2 \frac d{dt} \|\bw\|^2 +\nu \|A^{1/2}\bw\|^2 + (    B(\bw ,P_N\bu) ,\bw)
\\&= - \mu (    I_{h,N,\om}\bw ,\bw)+\mu( I_{h,N,\om}Q_N\bu ,\bw)
+(B(Q_N\bu, \bu) +B(P_N\bu,Q_N\bu) ,\bw),
}
where we made some obvious simplifications. As in Section \ref{sec.existence} we have 
\EQ{
- \mu (   I_{h,N,\om}\bw ,\bw) =  \mu (  \chi_\om  \bw-  I_{h,\om } ( \bw)   ,\bw)   - \mu\int_\om |\bw|^2\,d\bx,
}
and 
\EQ{\label{ineq.interp.1}
| \mu (   \chi_\om  \bw -I_{h ,\om}( \bw)   ,\bw) |&\leq  C\mu c_0  h^2 \|A^{1/2}\bw  \|_{L^2(\Om)}^2  +\frac \mu {4 } \|  \bw  \|_{L^2(\om)}^2.
}
The H\"older, Ladyzhenskaya and Young inequalities lead to 
\EQ{
|(B(\bw,P_N\bu),\bw)|\leq \frac {C} {\nu} \| A^{1/2}\bu\|^2\|\bw\|^2 +\frac \nu 4 \|A^{1/2}\bw\|^2.
} 
Applying the spectral inequality \eqref{ineq.spectral.torus.3} we obtain 
\EQ{
 \frac {C} {\nu} \| A^{1/2}\bu\| ^2\|\bw\| ^2 \leq  C \nu \lambda_1 G^2   C_{\om} e^{C_\om \sqrt{N}} \|\bw\|_{L^2(\om)} ^2. 
}
Also by standard interpolation inequalities, we have
\begin{align*}
&(B(Q_N\bu, \bu) +B(P_N\bu,Q_N\bu) ,\bw) 
\\&\leq C  (\|Q_N\bu\|^{1/2}\|A^{1/2}Q_N\bu\|^{1/2}  \|\bu\|^{1/2}\|A^{1/2}\bu\|^{1/2}   \\ 
  & \qquad + \|P_N\bu\|^{1/2}\|P_NA^{1/2}\bu\|^{1/2}   \|Q_N\bu\|^{1/2}\|Q_NA^{1/2}\bu\|^{1/2}      )  \|A^{1/2}\bw\|
\\&\leq C    \|Q_N\bu\|^{1/2}\|A^{1/2}Q_N\bu\|^{1/2}  \|\bu\|^{1/2}\|A^{1/2}\bu\|^{1/2}   \|A^{1/2}\bw\|
\\&\leq    \frac {C} {\nu^2} \|Q_N\bu\|\|A^{1/2}Q_N\bu\|   \|\bu\|\|A^{1/2}\bu\| +\frac \nu 4 \|A^{1/2}\bw\|^2\;.
\end{align*}
Note that
\EQ{
\frac {C} {\nu^2} \|Q_N\bu\|\|A^{/2}Q_N\bu\|   \|\bu\|\|A^{1/2}\bu\| &\leq \frac C {\nu^2\lambda_N^{1/2}}  \|A^{1/2} \bu\|^3 \|\bu\|
\\&\leq \frac C {\nu^2\lambda_N^{1/2}}( 2\nu^2\lambda_1 G^2 )^{3/2} (2\nu^2 G^2)^{1/2}
\\&\leq \frac C {\lambda_N^{1/2}} \nu^2 \lambda_1^{3/2}G^{4},
}
which can be made small by taking $N$ large.
Finally we have 
\EQ{
\mu|( I_{h,N,\om}Q_N\bu ,\bw)| &\leq  \mu \|I_{h,\om}Q_N\bu\|  \|\chi_\om\bw\|\\
&\leq  \mu \|\chi_\om Q_N\bu\| \|\bw \chi_\om\| \leq C \mu \|\chi_\om Q_N\bu\|^2 + \frac \mu 4 \| \bw\|_{L^2(\om)}^2.
}
By our assumption on  uniform Gevrey bounds for $\bu$ we have 
\[
  \mu \|\chi_\om Q_N\bu\|^2 \leq  C \mu \sum_{\sqrt N\lesssim |\bk|} |\hat \bu_\bk|^2\leq 
  C\mu \sum_{\sqrt N\lesssim |\bk|}
\frac {M^2} { \sqrt{2\pi |\Om|}} \bigg|\frac {L} \bk\bigg|^2 e^{-4\pi \sigma\big|\frac {\bk} L\big|},
\]
and this bound holds uniformly in time for sufficiently large times. We will return to this term later after specifying a connection between $\mu$ and $N$.

Collecting the above estimates and dropping terms where appropriate we obtain 
\EQ{\label{ineq.w.energy.2}
&\frac 1 2 \frac d{dt} \|\bw\|^2 +\frac \nu 2 \|A^{1/2}\bw\|^2  +\frac \mu 2 \int_\om |\bw|^2\,d\bx
\leq     C\nu \lambda_1 G^2 C_\om e^{C_\om \sqrt N} \|\bw\|_{L^2(\om)}^2 \\ &+ C \mu c_0 h^2 \|A^{1/2}\bw\|^2
+ C \mu \sum_{\sqrt N\lesssim |\bk|}
\frac {M^2} { \sqrt{2\pi |\Om|}} \bigg|\frac {L} \bk\bigg|^2 e^{-2\pi \sigma\big|\frac {\bk} L\big|} + \frac C {\lambda_N^{1/2}} \nu^2 \lambda_1^{3/2} G^4.
}

Provided $N$ is sufficiently large we have
\EQ{\label{cond.N.a}
 \frac C {\lambda_N^{1/2}} \nu^2 \lambda_1^{3/2} G^4 \leq \frac  {\bar \epsilon}2.
}
Assuming this holds, let
\[
\mu  = 2C\nu \lambda_1 G^2 C_\om e^{C_\om \sqrt N}\text{ and }h_*= \sqrt{ \frac \nu {4C \mu c_0}}
\]
and take $h\leq h_*$. Then \eqref{ineq.w.energy.2} simplifies to
\EQ{\label{ineq.w.energy.3}
\frac 1 2 \frac d{dt} \|\bw\|^2 +\frac \nu 4 \|A^{1/2}\bw\|^2   &\leq     
 C \mu \sum_{\sqrt N\lesssim |\bk|}
\frac {M^2} { \sqrt{2\pi |\Om|}} \bigg|\frac {L} \bk\bigg|^2 e^{-2\pi \sigma\big|\frac {\bk} L\big|} + \frac {\bar \epsilon} 2.
}
Using the definition of $\mu$ in terms of $N$ we obtain 
\EQ{
 C \mu \sum_{\sqrt N\lesssim |\bk|}
\frac {M^2} { \sqrt{2\pi |\Om|}} \bigg|\frac {L} \bk\bigg|^2 e^{-2\pi \sigma\big|\frac {\bk} L\big|}  &=  C\nu \lambda_1 G^2 C_\om  \sum_{\sqrt N\lesssim |\bk|}
\frac {M^2} { \sqrt{2\pi |\Om|}} \bigg|\frac {L} \bk\bigg|^2 e^{C_\om \sqrt N-2\pi \sigma\big|\frac {\bk} L\big|} 
 \\&\leq \frac {C\nu \lambda_1 G^2 C_\om M^2} {N} \sum_{\sqrt N\lesssim |\bk|}
 e^{C_\om \sqrt N-2\pi \sigma\big|\frac {\bk} L\big|},
}
where we have hidden some global parameters. Take $\delta_0$  large enough so that 
\[ \sum_{\sqrt N\lesssim |\bk|}
 e^{C_\om \sqrt N-2\pi \sigma\big|\frac {\bk} L\big|}\leq 1.\]
This is achieved if
\[
C_\om \lesssim \sigma.
\]
In addition to \eqref{cond.N.a}, we require that
\[
N\geq  \frac {2C\nu \lambda_1 G^2 C_\om M^2} {\bar \epsilon}.
\]
This leads to 
\EQ{\label{ineq.w.energy.4}
\frac 1 2 \frac d{dt} \|\bw\|^2 +\frac \nu 4 \|A^{1/2}\bw\|^2   &\leq   \bar \epsilon.
}
The Poincaré\'e inequality implies
\EQ{\label{ineq.w.energy.5}
&\frac d{dt} \|\bw\|^2 + \frac \nu {2}\lambda_1\|\bw\|^2   \leq    2 \bar \epsilon.
} 
By the Gronwall inequality we obtain
\[
\|\bw(t)\|^2\leq \|\bu_0\|^2 e^{-\nu \lambda_1 t /2} + \frac  {4 \bar \epsilon}  {\nu\lambda_1} (1 - e^{-\nu \lambda_1 t /2}) \leq \|\bu_0\|^2 e^{-\nu \lambda_1 t/2} +\frac {  \epsilon} 2,
\]
where we used the definition of $\bar \epsilon$ from the beginning of the proof. To complete the proof note that 
\[
\|   \bu(t)-\bv(t)\|^2 \leq \|\bw(t)\|^2 + \|Q_N\bu(t)\|^2\leq  C \|\bu_0\|^2 e^{-\nu \lambda_1 t/2} +  \frac 3 4  \epsilon,
\]
provided we take $N$ large enough so that
\[\|Q_{N}\bu\|^2\leq \frac\epsilon 4.\]
This plainly implies
\[
\|   \bu(t)-\bv(t)\| <\epsilon,
\]
for $t$ sufficiently large.

The proof for $\mathcal I_{h,N,\om}$ is similar but we need to modify our treatment of \eqref{ineq.interp.1} as we did at the end of the proof of Lemma \ref{lem.da.exist}.
\end{proof}

\begin{rmk}\label{rmk.bounded.domain}
Essentially the  same proof goes through for bounded domains if we use \eqref{ineq.spectral} and assume that  $\bu\cdot \phi_N$ decay sufficiently fast in $N$. The other modifications are standard \cite{AOT}.
\end{rmk}

\begin{proof}[Proof of Corollary \ref{thrm.exact.convergence.1}] 
Let $\bv$  be as in Theorem \ref{thrm.main1}.  Inspecting the proof of Theorem \ref{thrm.main1} we see that  
$$
\|\bv\|\leq \|\bv-P_N\bu\|+\|P_N\bu\|<  \|\bu_0\|^2 e^{-\nu \lambda_1 t/4} +\epsilon +\|\bu\|\;.
$$ 
This is an upper bound for $\bv$ that is   time-global and  independent of $\mu$ and $N$. For the same reason we get control of $\int_0^T \int_\Om |\bv|^2\,d\bx\,dt$ for finite $T$.  In contrast, the corresponding upper bounds obtained in Lemma \ref{lem.da.exist} depended on $\mu$ and $N$. This implies for any $0<\epsilon\ll 1$ we can construct a solution $\bv_\epsilon$ for parameters $N_\epsilon$, $\mu_\epsilon$ and $h_\epsilon$ to \eqref{eq.da1} with the usual energy class bounds holding independently of $N_\epsilon$ and $\mu_\epsilon$, provided we have knowledge of $\bu$ at all points in $\om$.   As $\epsilon\to 0$, we have $N_\epsilon,\mu_\epsilon \to \I$ while $h_\epsilon\to 0$.  By Banach-Alaoglu, we have that there exists $\bvv$ so that $\bv_\epsilon\to \bvv$ in the weak-star topology on $L^\I( [0,T];L^2)$ for every $T>0$ as well as the weak topology on $L^2([0,T];H^1)$. 

%
%

Fix a measurable set $U$. Let $\Delta>0$ be a given time scale. Then for any $t$,
\[
\int_{t}^{t+\Delta}\int_U (\bu-\bvv) \,d\bx\,ds =   \int_{t}^{t+\Delta } \int_{U} (\bu-\bv_\epsilon) \,d\bx \,ds +\int_{t}^{t+\Delta } \int_{U} (\bv_\epsilon-\bvv) \,d\bx \,ds.
\]
We have by Theorem \ref{thrm.main1} that
\EQ{
 \int_{t}^{t+\Delta} \int_{U} (\bu-\bv_\epsilon)\,d\bx \,ds  &\leq |U|^{1/2}\bigg( \sup_{s\in [t,t+\Delta ]} |\bu-\bv_\epsilon|^2( s) \bigg)^{1/2}
\\&\leq |U|^{1/2}\bigg( \sup_{s\in [t,t+\Delta ]} |\bu_0|^2 e^{-\nu \lambda_1 t/4} +\frac 3 4 \epsilon \bigg)^{1/2}.
}
Additionally we know that 
\[
\bigg|\int_{t}^{t+\Delta } \int_{U} (\bv_\epsilon-\bvv)(\bx,t) \,d\bx \,ds \bigg|\to 0,
\]
by $*$-weak convergence in $L^\I L^2$. Hence we may choose $\epsilon$ so that $\epsilon < e^{-t}$ and the above quantity  is smaller than $e^{-t}$. This gives the advertised exponentially decaying bound.

For the second item in Corollary \ref{thrm.exact.convergence.1}. By the Lebesgue differentiation theorem, for almost every $t$ we have 
\[
\lim_{\Delta t \to 0} \frac 1 {\Delta t} \int_{t}^{t+\Delta t} \int_{U} (\bu-\bvv) (\bx,t)\,d\bx\,ds =  \int_{U} (\bu-\bvv)(\bx,t) \,d\bx,
\]
where $\Delta t$ is a time-scale that depends on $t$.  
Let $S$ denote the set of times for which this holds. Then $|S^c|=0$ where $|\cdot |$ denotes Lebesgue measure on the line.
Fix $t>0$. Then for $\Delta t$ sufficiently small we have 
\EQ{
 &\int_{U} (\bu-\bvv)(\bx,t) \,d\bx 
\\&\leq  \frac 1 {\Delta t} \int_{t}^{t+\Delta t} \int_{U} (\bu-\bvv)(\bx,t) \,d\bx \,ds  +  e^{-t}
\\&\leq   \frac 1 {\Delta t} \int_{t}^{t+\Delta t} \int_{U} (\bu-\bv_\epsilon)(\bx,t) \,d\bx \,ds  + \frac 1 {\Delta t} \int_{t}^{t+\Delta t} \int_{U} (\bv_\epsilon-\bvv)(\bx,t) \,d\bx \,ds + e^{-t}.
}
We have already explained how the first two terms can be made exponentially small.
Since this holds for all $t\in S$ we see that 
\[
\chi_S(t) \int_{U} (\bu-\bvv)(\bx,t) \,d\bx\to 0,
\]
at an exponential rate. 

\end{proof}

\begin{rmk}
The precise dynamics of $\bvv$ are unclear  because we have not obtained a governing system for $\bvv$ via the  limiting process. The challenge to doing so is that, as $\mu_\epsilon\to \I$, so does $\mu_\epsilon I_{h_\epsilon,N_\epsilon,\gamma} (\bu)$.  To make sense of the equations after taking limits, would require $\mu_\epsilon P_{N_\epsilon}I_{h_\epsilon}(\bu-\bv^\epsilon)$ is bounded in some sense. Because $\bv^\epsilon \to \bvv$, we would need $\bu = \bvv$ for such a bound. In this case, $\bvv$ recovers the flow exactly for all times, not just as $t\to \I$.  Even granting this, the rate of convergence of $\bv^\epsilon \to \bvv$ would need to be rapid enough to compensate for the exponential growth of $\mu_\epsilon$.
\end{rmk}
   
\begin{rmk}
\label{rmk.da2} We now sketch the proof of Theorem \ref{thrm.main2}. Granted existence, the proof follows the argument of Azouani-Olson-Titi in  \cite{AOT} except we do not have a positive global term originating from the interpolant. Instead we have 
\EQ{\label{6.22.20.a}
\mu \int_\om |\bw|^2 \,d\bx. 
}
This and diffusion are used to hide 
\[
\frac 1 {2\nu} \|\bu\|^2|\bw|^2.
\]
Indeed, by the spectral inequality \eqref{ineq.spectral.bounded} from \cite{XiYo} we have 
\[
\frac 1 {2\nu} \|\bu\|^2|\bw|^2\leq c\nu \lambda_1 G^2\bigg( \frac k {\lambda_1(\Om\setminus \om)} \|\bw\|_{L^2(\om)}^2  + \frac 1 {\lambda_1(\Om\setminus \om)} \|\bw\|^2   \bigg),
\]
where $\lambda_1(\Om\setminus\om)$ is the Poincar\'e constant for the domain $\Om\setminus \om$.  We require $\Om\setminus \om$ to be thin enough that 
\EQ{\label{6.22.20.b}
 \frac {c\lambda_1 G^2} {\lambda_1(\Om\setminus \om)} \sim 1,
}
as this will allow us to hide the $H^1$ term in the diffusion. We then choose $\mu$ large so that the local quantity is absorbed by \eqref{6.22.20.a}. Plainly this will allow us to execute the remainder of the argument from \cite{AOT}.

We now analyze our choice of parameters. 
Let $h_0$ be  the thickness of $\Om\setminus \om$. Then $h_0$ is roughly $\lambda_1(\Om\setminus \om)^{-1/2}$. Putting this in \eqref{6.22.20.b} gives
\[ 
h_0 \sim  \frac 1 {\sqrt \lambda_1 G}.
\]
This is on the order of the length scale of the global grid in \cite{AOT}. Hence, Theorem \ref{thrm.main2} says that it is possible to ignore roughly the outer band of observables in the volume-elements global data assimilation algorithm and still exactly recover the solution.

\end{rmk}

 \section{Computational results} \label{sec.comp}

Our computations were done on the NSE in vorticity form with a fully dealiased pseudospectral code corresponding to $N\times N$ nodal values in the physical space $\Om=[0,2\pi]^2$.  The force $\bbf$, specified in Fourier space, was the same as that used in \cite{OT1,OT2}, time independent and concentrated on the annulus with wave numbers  $10\le |\bk| \le 12$.  The reference solution was evolved from a zero initial value  for 25,000 time units at which point the energy, enstrophy and palinstrophy have all settled into time series for a chaotic solution (see Figure \ref{chaosfig}) with steady statistics.  The viscosity was set to $\nu=10^{-4}$, and a scalar multiple on the force is chosen so that the Grashof number $G=10^6$.  
Both the vorticity of the reference solution $\omega_N=\nabla \times\bu_N$ and that of the synchronizing solution $\tilde\omega_N=\nabla \times\bv$ were solved  using the third-order Adams-Bashforth method in \cite{OT1,OT2} in which the linear term is handled exactly through an integrating factor.  The step size was $\Delta t=0.01$ with $N=512$, consistent with \cite{OT2} at this Grashof number.   We took data on square subdomains $\om=\om_j$, $j=1,2,3,4$, $\om_{j} \subset \om_{j-1}$, each centered in $\Om$ and with relative size 
$$
|\Omega_1|=0.7656|\Om|\;, \quad |\Omega_2|=0.6602  |\Om|\;,  \quad
 |\Omega_3|=0.5265 |\Om|\;, \quad |\Omega_4|=0.2500|\Om| \;.
$$
\begin{figure}[ht]
\psfrag{t}{\tiny$t$}
\psfrag{enstrophy}{\tiny enstrophy}
 \centerline{\includegraphics[scale=.35, angle=-90 ]{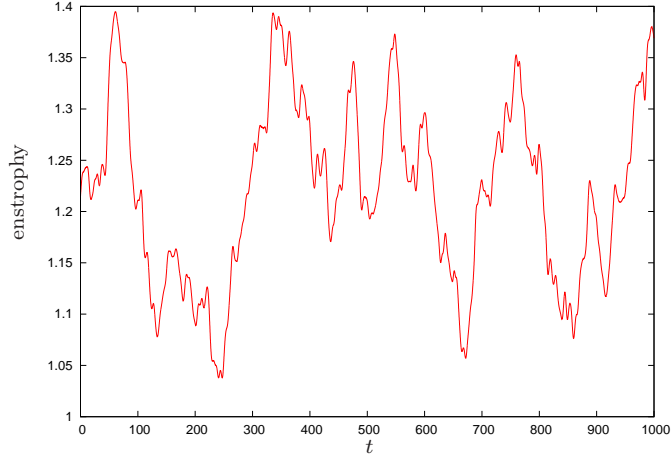}  } \vskip .25 truein
\caption{Time series of enstrophy, $\|\omega_N\|^2_{L^2(\Om)}$, indicating chaos.}
\label{chaosfig}
\end{figure}

 An interpolating operator $J$ was computed  by first applying an FFT$^{-1}$ to the Fourier coefficients of $\tilde\omega_N-\omega_N$.  In order to use coarse data, the resulting difference was used at only every  $2^p$-th node in each direction, with results compared for $p=1,2,3,4$, so that $h=\pi/128, \pi/64, \pi/32$ and $\pi/16$, respectively.
Before transforming back via an FFT, the field within the subdomain $\om$ was smoothened by the recursive averaging operator $\interp_p$ depicted in Figure \ref{smoothfig} and set to zero on $\Om\setminus \om$ so that
$$
J(\tilde\omega_N-\omega_N)=\text{FFT} \circ \chi_\om\circ\interp_p\circ\text{FFT}^{-1}(\tilde\omega_N-\omega_N)\;.
$$
Note that the final transformation by the FFT serves to filter, just as $P_N$ did in our analysis sections, though $N$ is no longer the number of Fourier modes.
After some experimentation we found that taking the relaxation parameter $\mu=50$ to be near optimal under these conditions. 
\begin{figure}[ht]
\psfrag{a}{\tiny$a$}
\psfrag{b}{\tiny$b$}
\psfrag{c}{\tiny$c$}
\psfrag{d}{\tiny$d$}
\psfrag{ab}{\tiny$\frac{a+b}{2}$}
\psfrag{bc}{\tiny$\frac{b+c}{2}$}
\psfrag{cd}{\tiny$\frac{c+d}{2}$}
\psfrag{ad}{\tiny$\frac{a+d}{2}$}
\psfrag{abcd}{\tiny$\frac{a+b+c+d}{4}$}
 \centerline{\includegraphics[scale=.5 ]{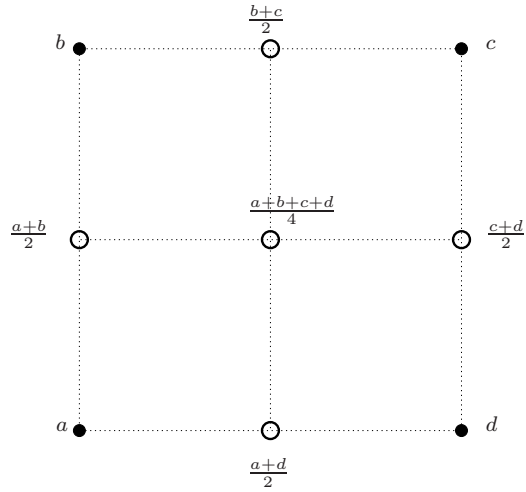}}
\caption{First recursive step of $\interp_p$.  Values of $\tilde\omega_N-\omega_N$ are $a$, $b$, $c$, $d$ at the corners.}
\label{smoothfig}
\end{figure}

We begin by testing the effect of the size of the subdomain.
 The relative $L^2$ and $L^\infty$ norms are compared in Figure \ref{L2fig} with the resolution of data fixed $h=\pi/32$ ($p=3$). To be clear,
these errors are measured as 
$$
\frac{\|\tilde\omega_N-\omega_N\|_{L^2(\Om)}} {\|\omega_N\|_{L^2(\Om)}}, \qquad
\frac{\max_{0\le j, k \le N-1}  |(\tilde\omega_N-\omega_N) (x_j,y_k)|}
{\max_{0\le m, n \le N-1} |\omega_N (x_m,y_n)|}\;,
$$
respectively.  Machine precision is reached for data collected over $\om_1$ in 1000 time units. By then, in the case of $\om_2$, the error is within $10^6$, while for $\om_3$, it has barely budged from unity.

We next vary the resolution of the observed data for two subdomains, $\om_2$ and $\om_3$ in Figure \ref{L2k1234}.
Over both subdomains there is little difference between the relative $L^2$ errors for $p=1,2,3$.  The resolution associated with $p=4$ ($h=\pi/16$) appears to be too coarse for nudging over $\om_2$, when measured in this sense.  Likewise, finer resolution in the case of $\om_3$ does not indicate convergence to the reference solution, at least by 1000 time units. While the relative $L^2$ error for $p=4$ suggests little resemblance between $\tilde\omega_N$ and $\omega_N$, particularly in the case of $\om_3$, we see from the vorticity field plots in Figure \ref{om3fields} that the main spatial features over the full domain $\Om$ are nevertheless captured. 
\begin{figure}[ht]
\psfrag{t}{\tiny$t$}
\psfrag{Omega}{\tiny$\Omega$}
\psfrag{omega}{\tiny$\omega$}
\psfrag{omegam}{\tiny$\omega^-$}
\psfrag{omegap}{\tiny$\omega^+$}
\psfrag{'L2if64r' using 1:2}{\qquad \tiny$\Omega_3$}
\psfrag{'L2if48r' using 1:2}{\qquad\tiny$\Omega_2$}
\psfrag{'L2if32r' using 1:2}{\qquad\tiny$\Omega_1$}
\psfrag{'L2if64' using 1:3}{\qquad \tiny$\Omega_3$}
\psfrag{'L2if48' using 1:3}{\qquad\tiny$\Omega_2$}
\psfrag{'L2if32' using 1:3}{\qquad\tiny$\Omega_1$}
\psfrag{G=1m, mu=50, N=512, k=3}{}
\psfrag{L^2 rel. error}{\tiny $L^2$ rel. error}
\psfrag{L^inf rel. error}{\tiny $L^\infty$ rel. error}
\psfrag{Linf}{\tiny $L^\infty$ rel. error}
  \centerline{\includegraphics[scale=.35, angle=-90]{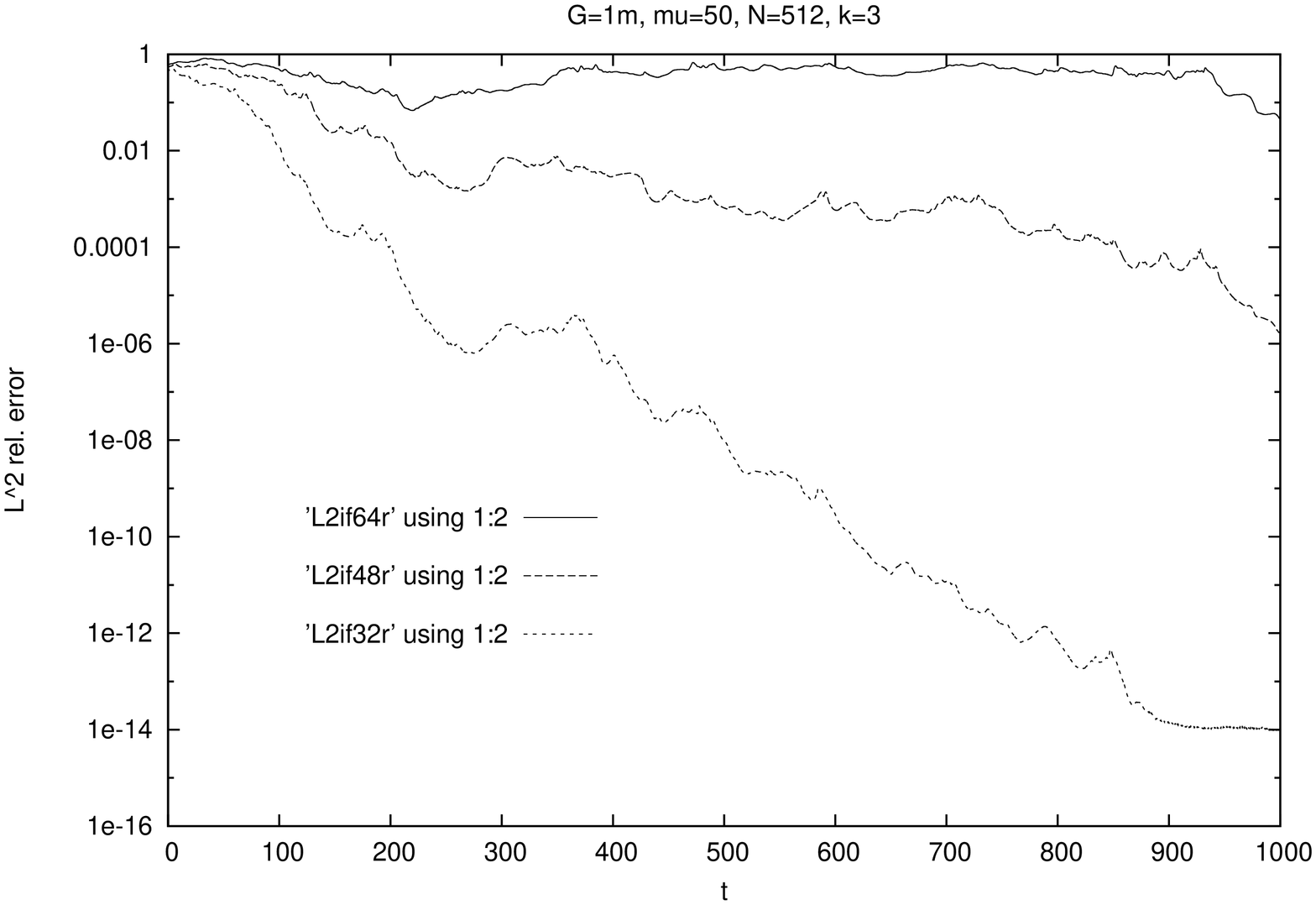} \qquad \includegraphics[scale=.35, angle=-90]{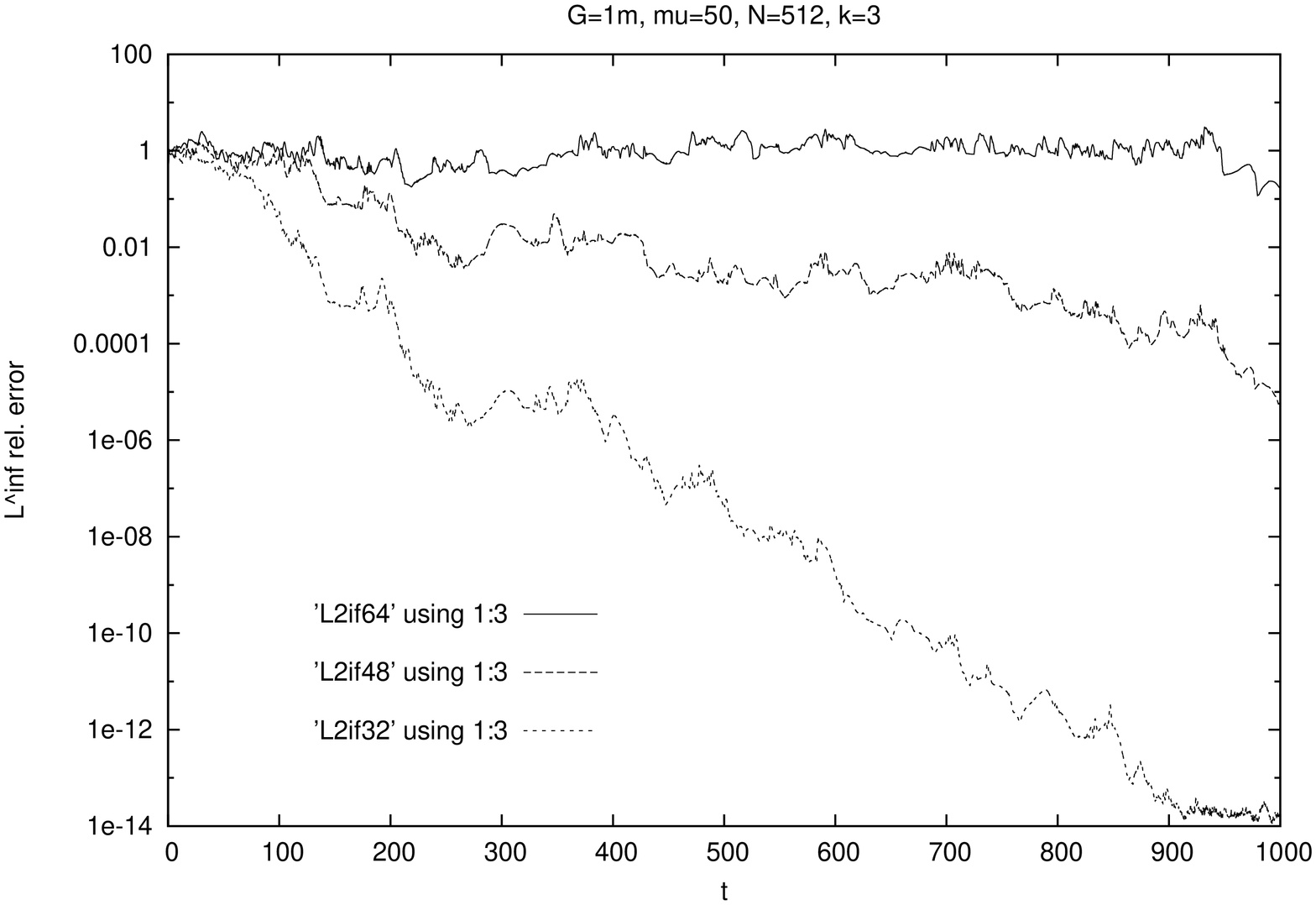} }
\vskip .15 truein
\caption{Relative $L^2$ and $L^\infty$ error for $\om=\om_1,\om_2, \om_3$, $p=3$ ($h=\pi/32$).}
\label{L2fig}
\end{figure}

\begin{figure}[ht]
\psfrag{t}{\tiny$t$}
\psfrag{Omega}{\tiny$\Omega$}
\psfrag{omega}{\tiny$\omega$}
\psfrag{48k4abcd}{\tiny$p=4$}
\psfrag{48k3abcd}{\tiny$p=3$}
\psfrag{48k2abcd}{\tiny$p=2$}
\psfrag{48k1abcd}{\tiny$p=1$}
\psfrag{64k4abcd}{\tiny$p=4$}
\psfrag{64k3abcd}{\tiny$p=3$}
\psfrag{64k2abcd}{\tiny$p=2$}
\psfrag{64k1abcd}{\tiny$p=1$}
\psfrag{G=1m, mu=50, N=512, k=3}{}
\psfrag{G=1m, mu=50, N=512, k=4}{}
\psfrag{L^2 rel. error}{\tiny $L^2$ rel. error}
\psfrag{L^inf rel. error}{\tiny $L^\infty$ rel. error}
\psfrag{Linf}{\tiny $L^\infty$ rel. error}
  \centerline{\includegraphics[scale=.35, angle=-90]{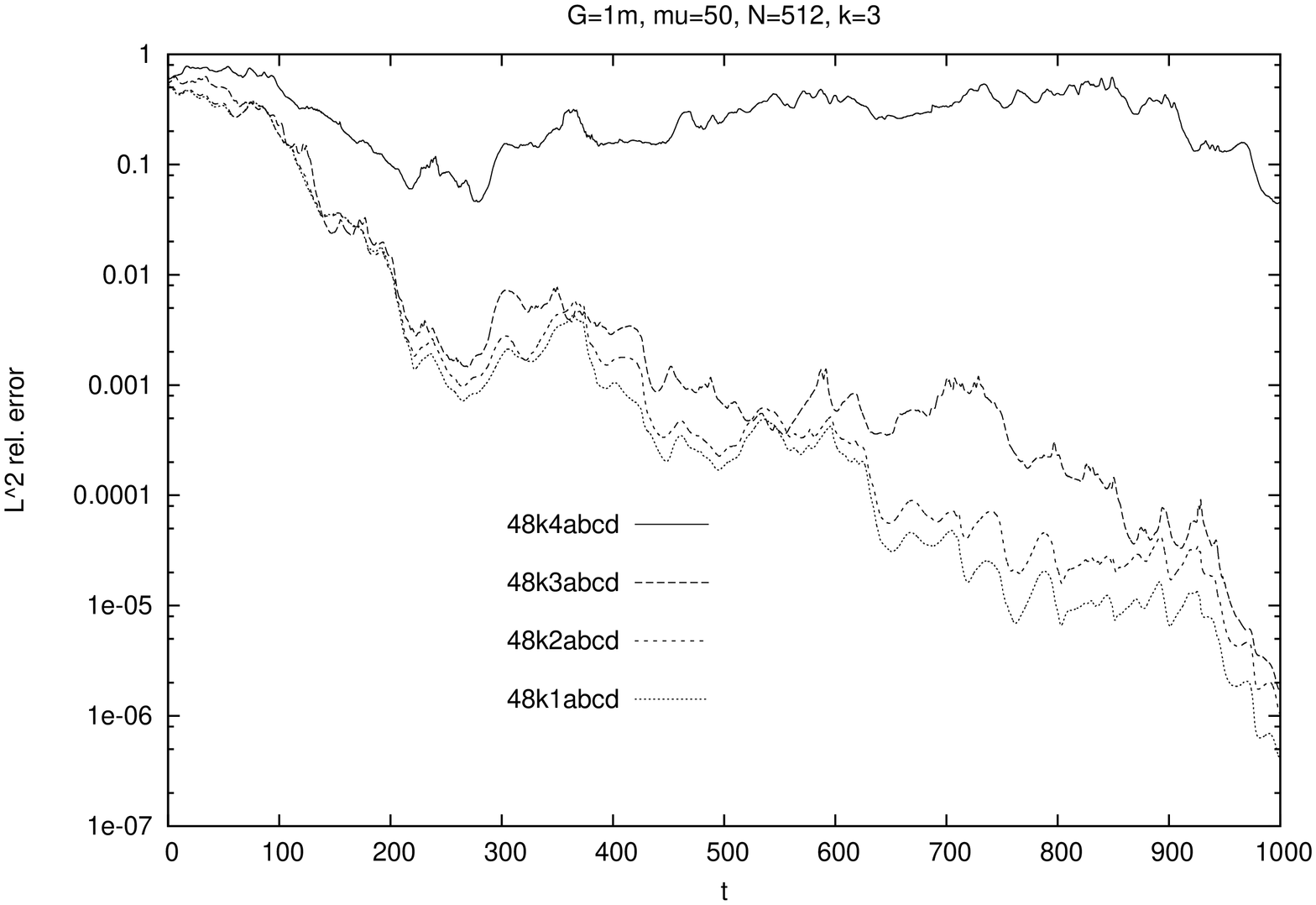} \qquad \includegraphics[scale=.35, angle=-90]{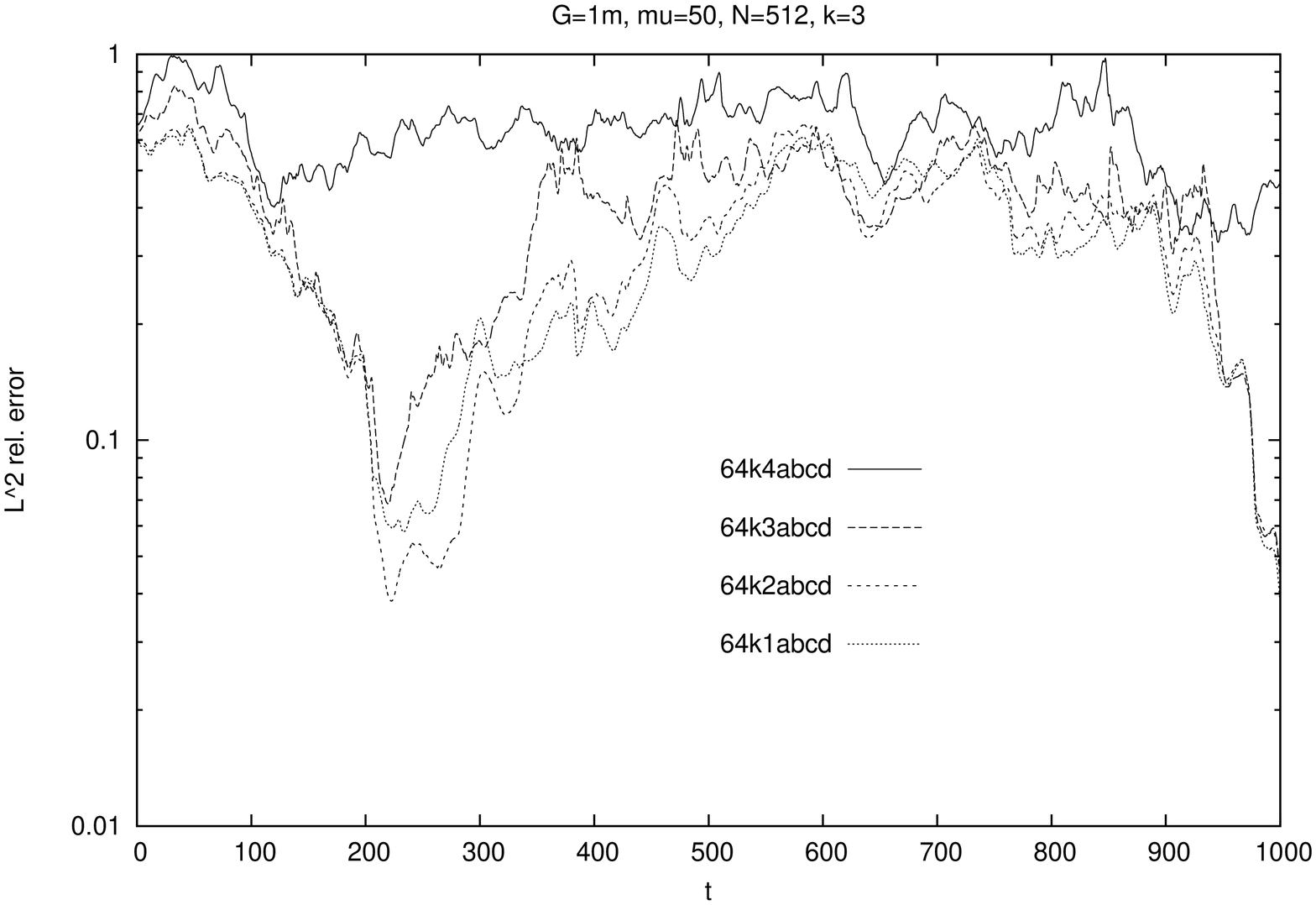} }
\vskip .15 truein
\caption{Relative $L^2$ error for $p=1,2,3,4$. Left: $\om=\Omega_2$, right: $\om=\Omega_3$.}
\label{L2k1234}
\end{figure}

\begin{figure}[ht]
\psfrag{x}{$x$}
\psfrag{y}{$y$}
  \vskip .5truein  
  \centerline{\includegraphics[scale=.6]{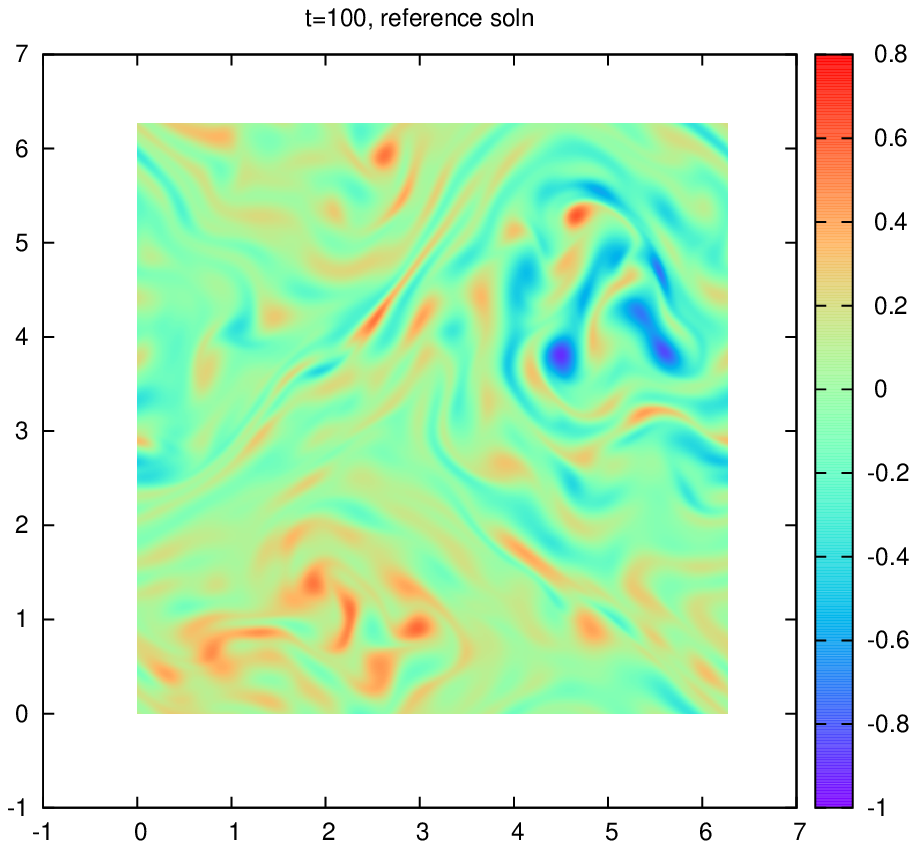} \quad\includegraphics[scale=.6]{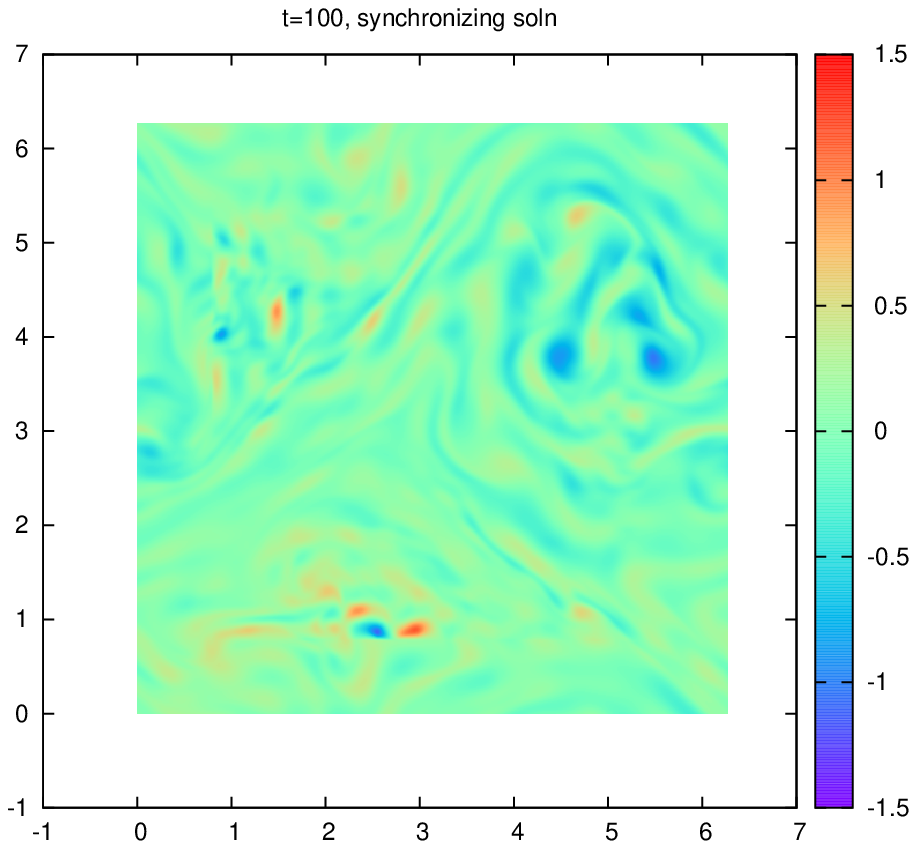} \quad  \includegraphics[scale=.6]{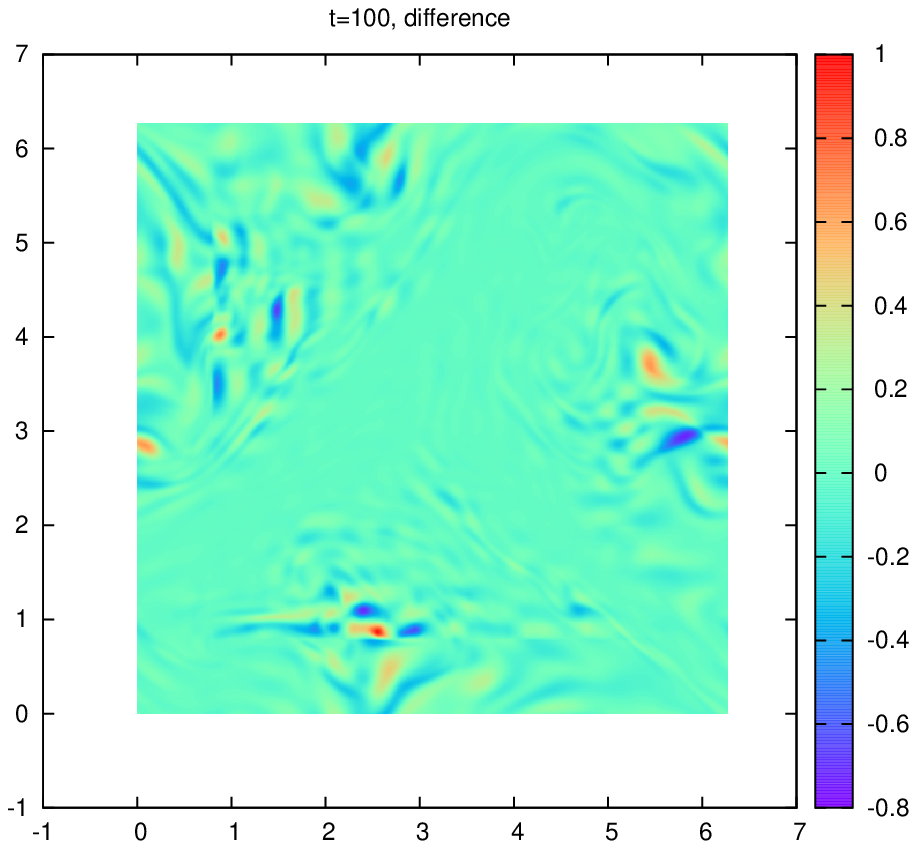}}
  \vskip .5truein  
  \centerline{\includegraphics[scale=.6 ]{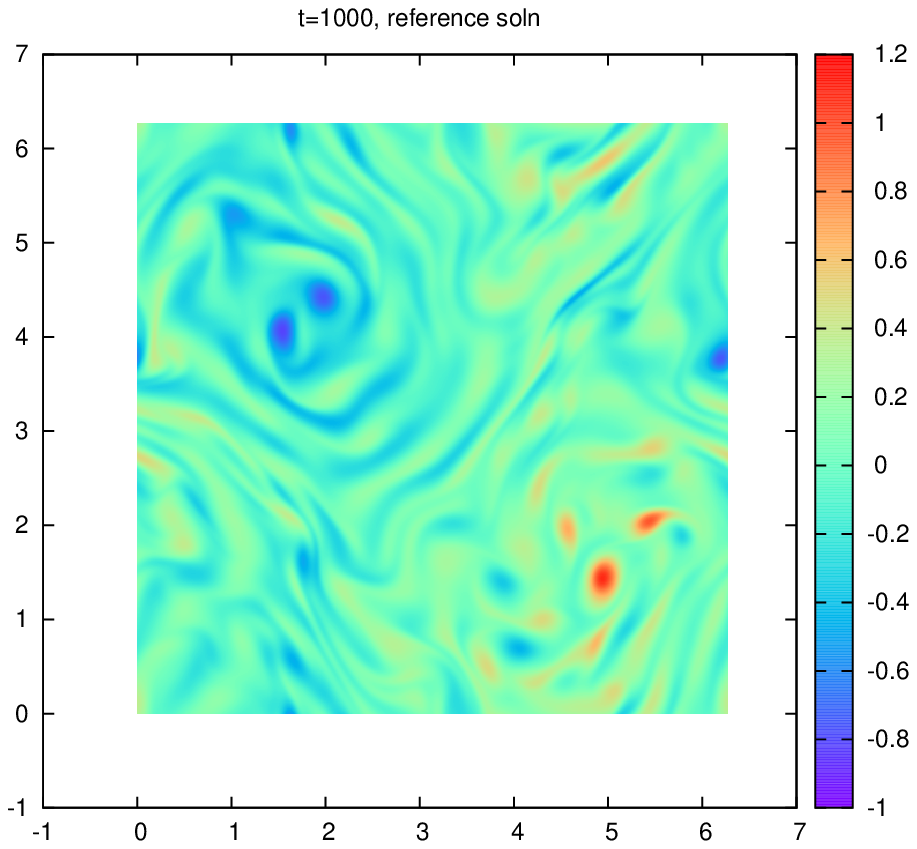} \quad \includegraphics[scale=.6]{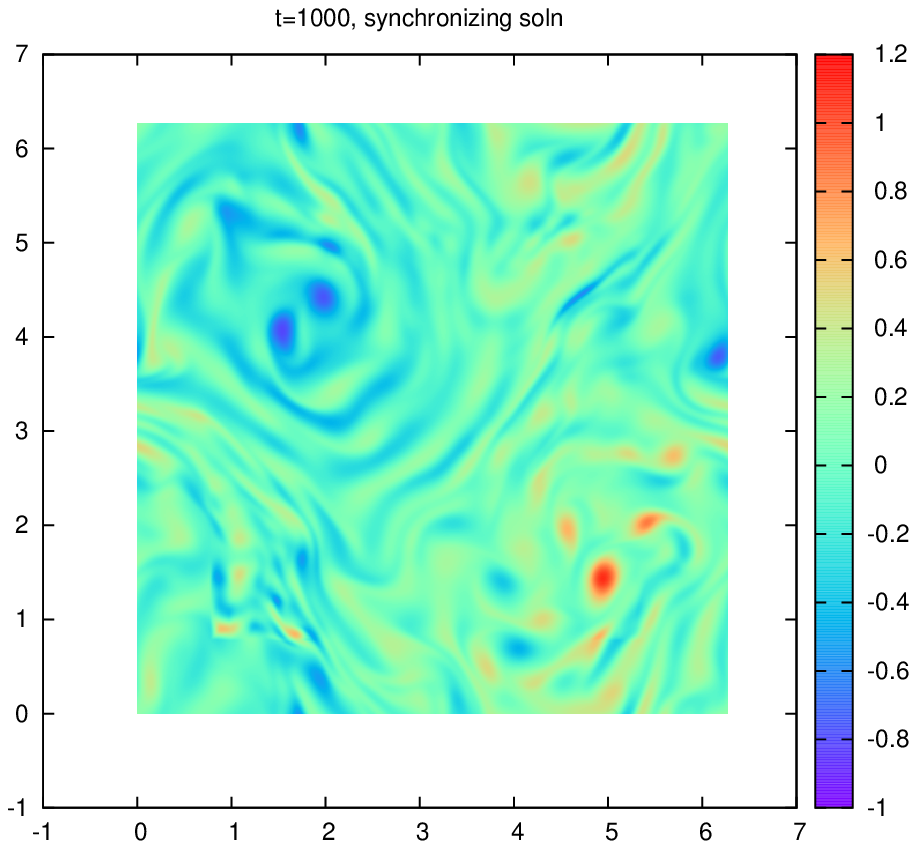}  \quad \includegraphics[scale=.6]{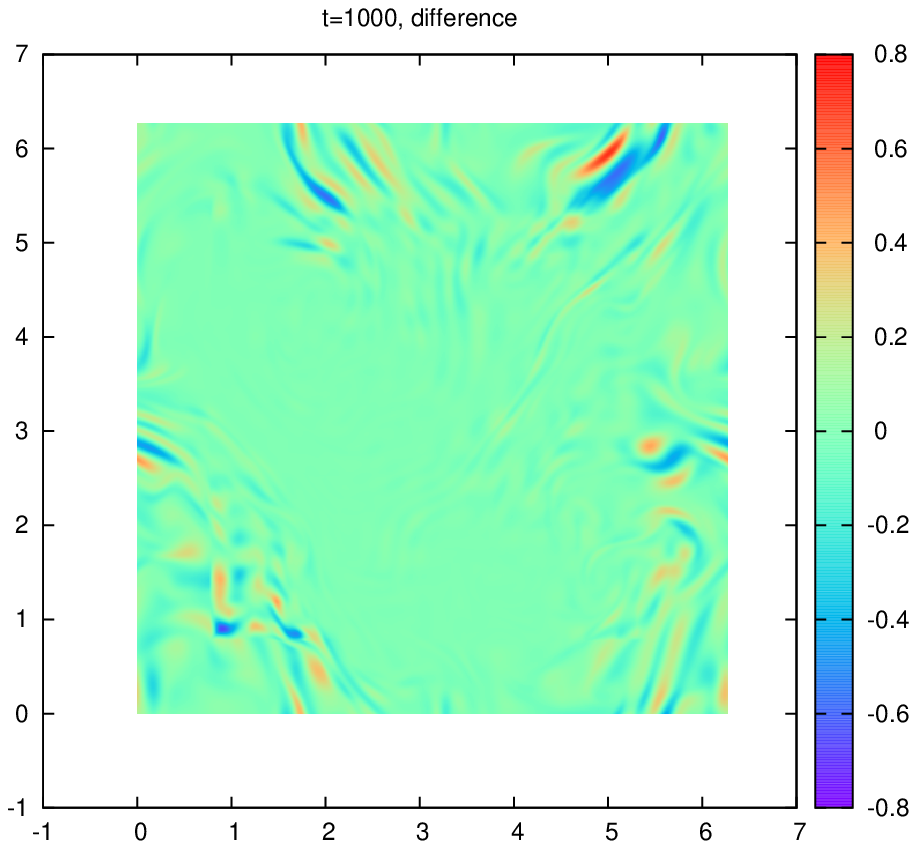}}
\caption{Snapshots of $\omega_N$, $\tilde\omega_N$ and difference, for $\om_3$, $p=4$. Top: $t=100$, bottom: $t=1000$. }
\label{om3fields}
\end{figure}

\begin{figure}[ht]
\psfrag{t}{\tiny$t$}
\psfrag{Omega}{\tiny$\Omega$}
\psfrag{omega}{\tiny$\omega$}
\psfrag{omegam}{\tiny$\omega^-$}
\psfrag{omegap}{\tiny$\omega^+$}
\psfrag{L^2 rel. error}{\tiny rel. error}
\psfrag{'L2Om0abcd}{\tiny$L^2(\Omega_0) $}
\psfrag{'L2Om3abcd}{\tiny$L^2(\Omega_3)$}
\psfrag{'L2Om1abcd}{\tiny$L^2(\Omega_4)$}
\psfrag{Linf}{\tiny $L^\infty$ rel. error}
  \centerline{\includegraphics[scale=.35, angle=-90]{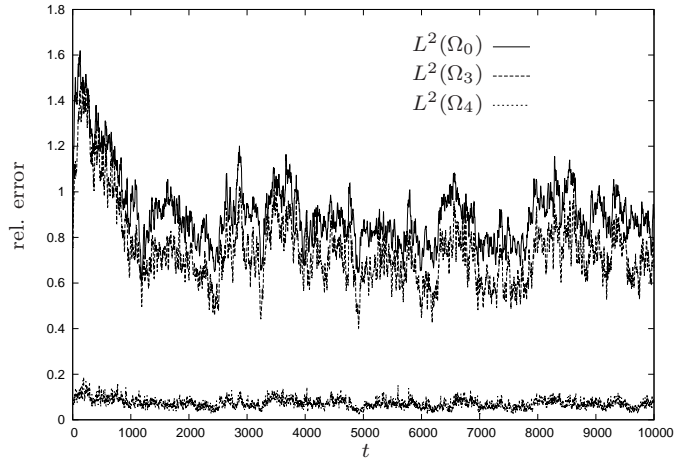} }
\vskip .15 truein
\caption{Relative $L^2$ error over various domains, data in $\Omega_4$, $h=\pi/128$.}
\label{L2Omfig}
\end{figure}


Since the $L^2$ errors over $\Om$ did not decay for data restricted to $\om_3$, we cannot expect them to do so when restricting to $\om_4$. We consider then relative $L^2$ errors that are measured also over subdomains. We found that even using data at every other node, the relative $L^2$ error over $\Omega_0$ is nearly unity, after nudging all the way to $t=10000$ (see Figure \ref{L2Omfig}). The relative $L^2$ error over $\Omega_3$ is nearly the same as that over $\Omega_0$.   From Figure \ref{fig10000} we see that despite the size of these errors, again the main features of the vorticity field emerge already at $t=1000$, but only to roughly the same extent at $t=10000$, consistent with the $L^2$ error.  We also note that while the relative $L^2$ error over $\Omega_4$, where the data is taken,
is roughly $0.1$, the plot of the difference $\tilde\omega_N-\omega_N$ within $\Omega_4$ is uniformly small.  

\begin{figure}[ht]
\psfrag{x}{$x$}
\psfrag{y}{$y$}
\vskip .5truein
  \centerline{\includegraphics[scale=.6 ]{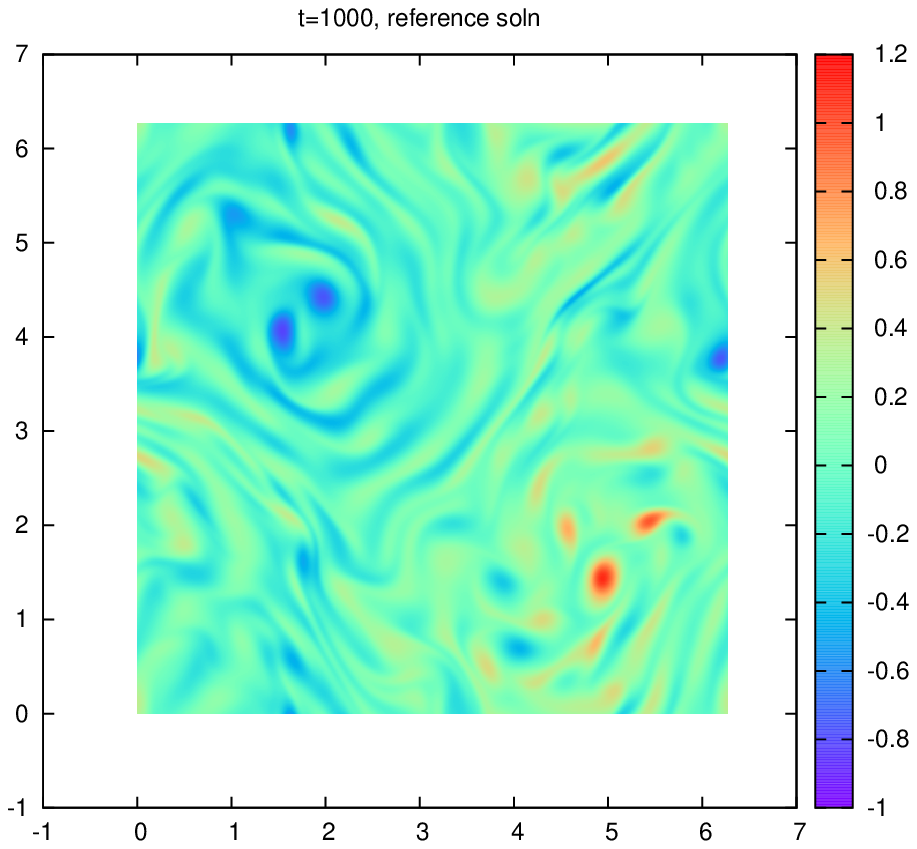} \quad \includegraphics[scale=.6]{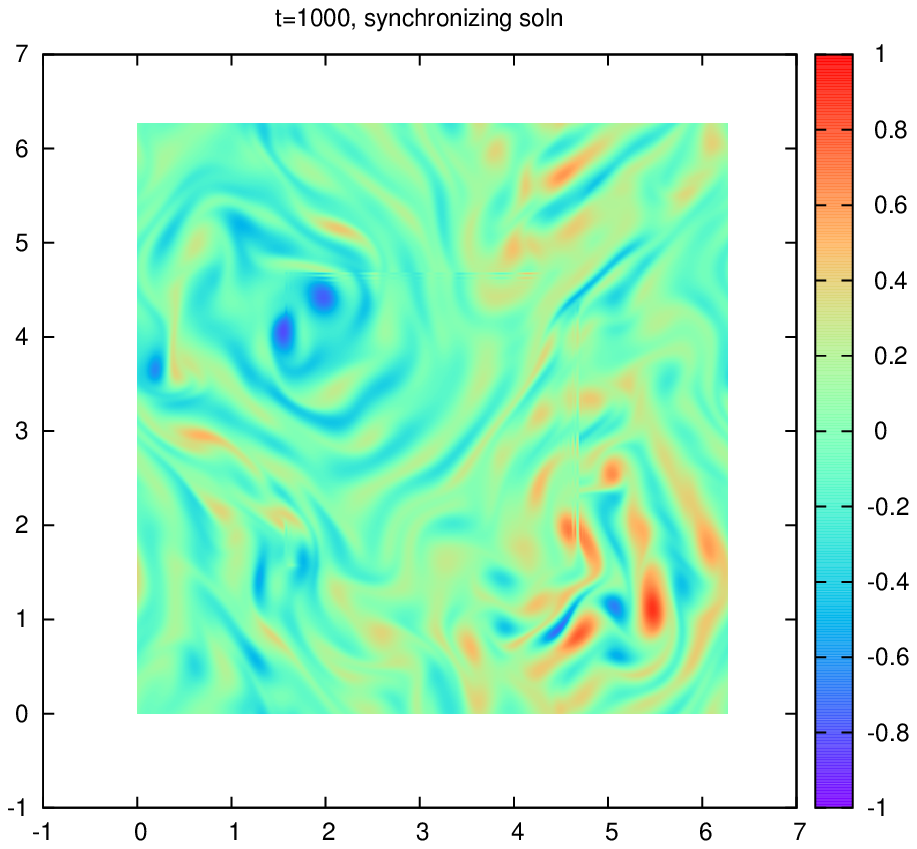}  \quad \includegraphics[scale=.6]{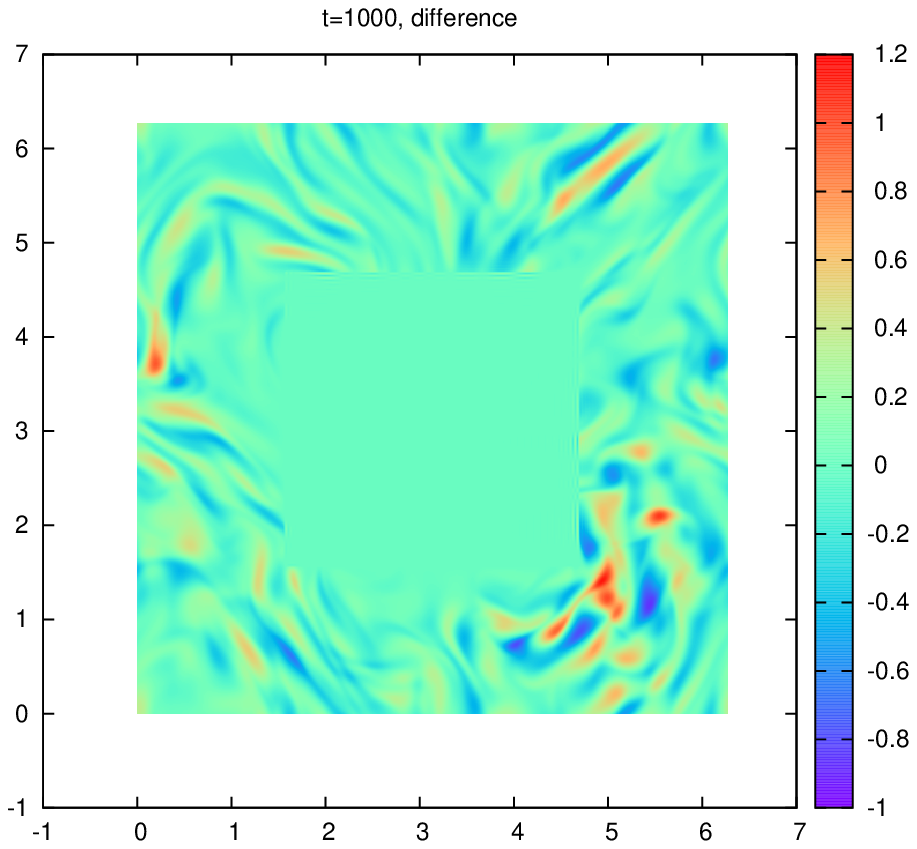}} \vskip .5truein
  \centerline{\includegraphics[scale=.6 ]{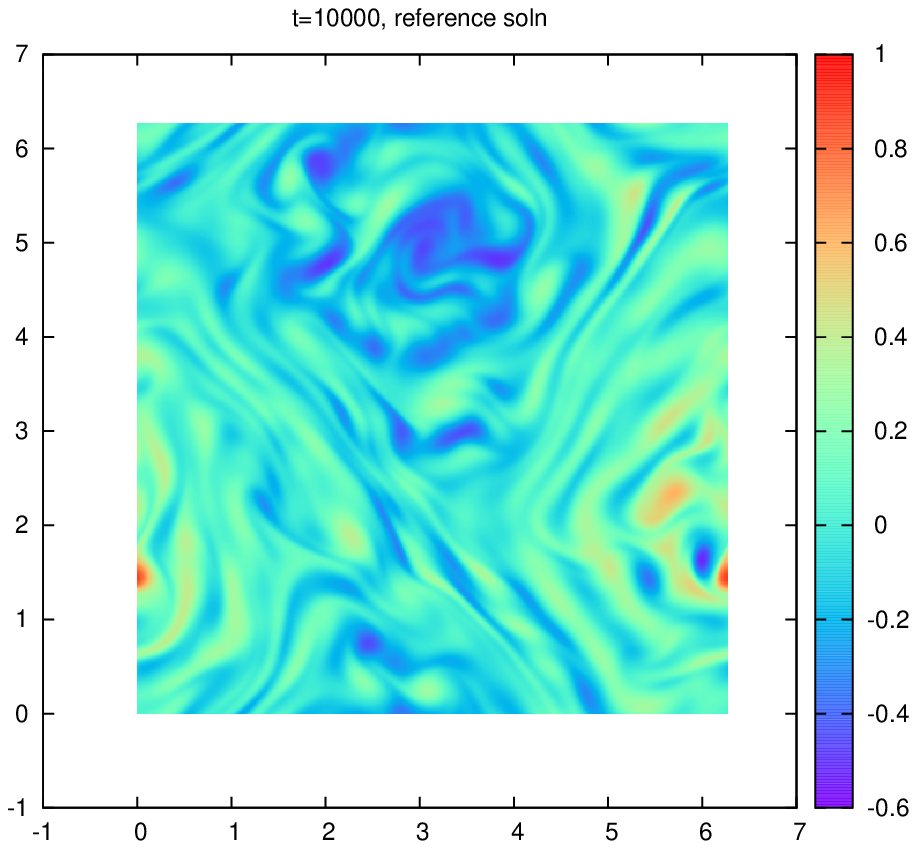} \quad \includegraphics[scale=.6]{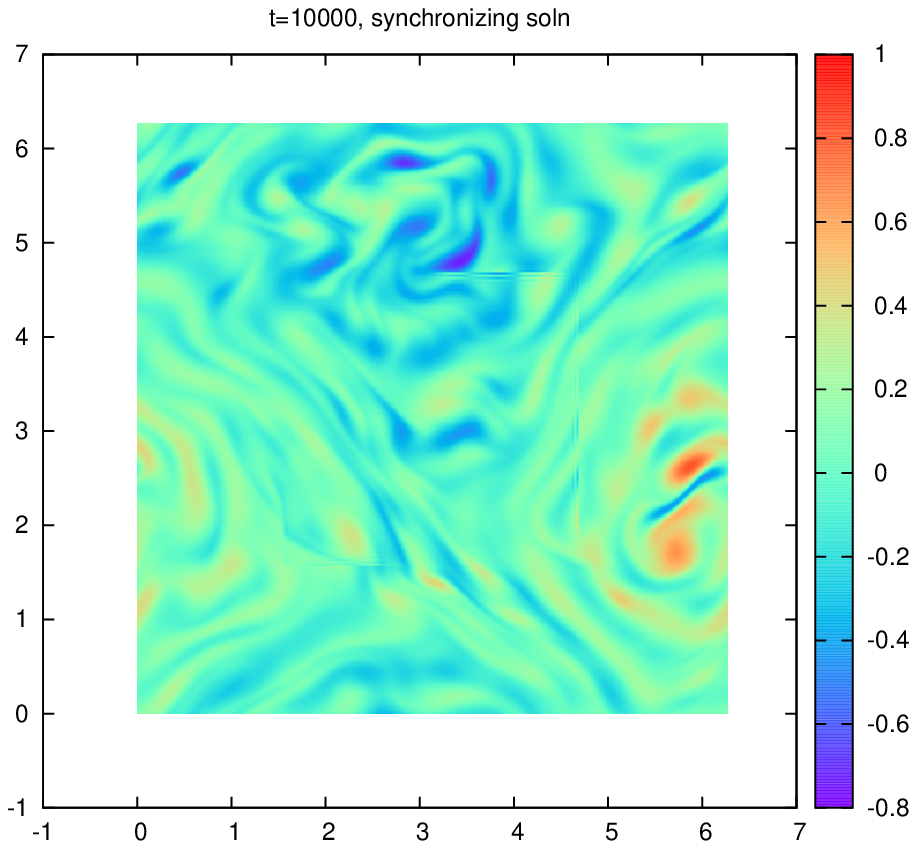}  \quad \includegraphics[scale=.6]{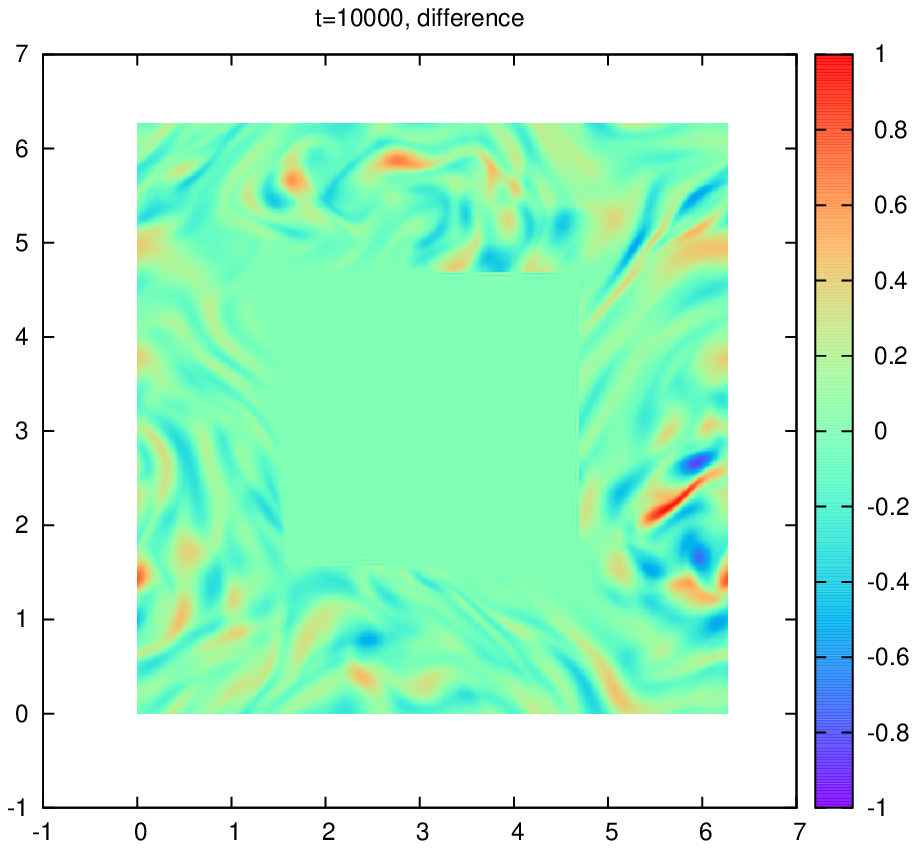}}
\caption{Snapshots of $\omega_N$, $\tilde\omega_N$ and difference, for data in $\om_4$, $p=1$.  Top: $t=1000$, bottom: $t=10000$.}
\label{fig10000}
\end{figure}

 \section{Mobile data} \label{sec.mobile}

Our emphasis to this point has been on how well nudging over a fixed subdomain can recover the reference solution over the {\it entire} computational domain.   The field plot of the difference in Figure \ref{om3fields}  (\ref{fig10000}) show that even with the coarsest data (smallest subdomain), the reference solution {\it within} the subdomain is captured well, despite the problem being global over $\Om$.

This leads us to consider moving the subdomain where the data is collected as the solution evolves.  We start with $\om_4(t)$, a subdomain with $1/4$-th the area of the computational domain, specified by the location of its lower left corner $(n_x,n_y)$ on the $N\times N$ discrete grid.  The movement of the subdomain is determined by the periodic extension of the functions shown in Figure \ref{cornerfig}.  The subdomain thus moves counterclockwise, covering the entire computational domain in one time unit.  We fix the local interpolating operator at our most coarse setting $h=\pi/16$ ($p=4$).  The results over the initial cycle in Figure \ref{figuuu} shows that synchronization is already well underway in just one time unit.    The relative errors are plotted in Figure \ref{errmobile}. Convergence to near machine precision is reached in one-tenth the time needed using finer data on the largest stationary subdomain, $\Omega_1$ (compare to Figure \ref{L2fig}).

\begin{figure}[ht]
\psfrag{1/4}{\tiny$1/4$}
\psfrag{1/2}{\tiny$1/2$}
\psfrag{3/4}{\tiny$3/4$}
\psfrag{1}{\tiny$1$}
\psfrag{ix}{\tiny$n_x(t)$}
\psfrag{iy}{\tiny$n_y(t)$}
\psfrag{x}{\tiny$x$}
\psfrag{y}{\tiny$y$}
\psfrag{t}{\tiny$t$}
\psfrag{N2}{\tiny$N/2$}
 \centerline{\includegraphics[scale=.6 ]{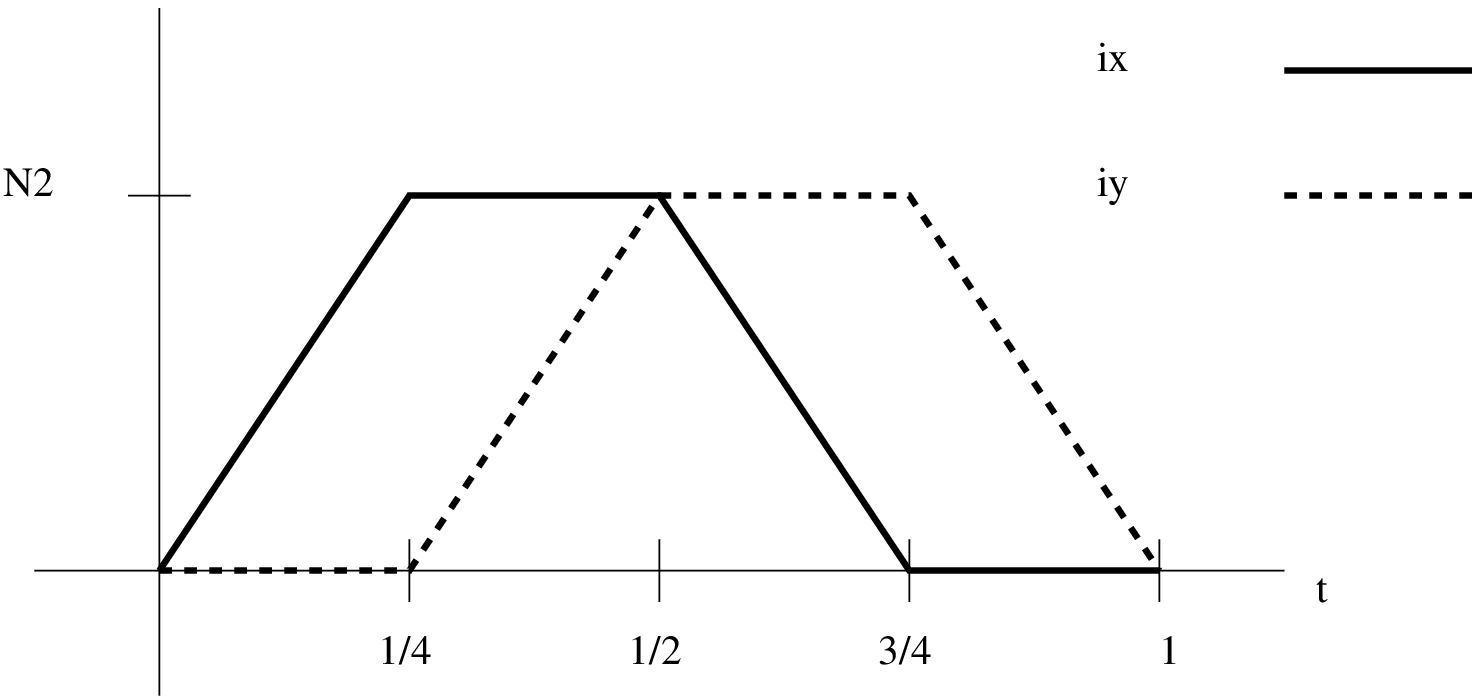} \qquad \includegraphics[scale=.4]{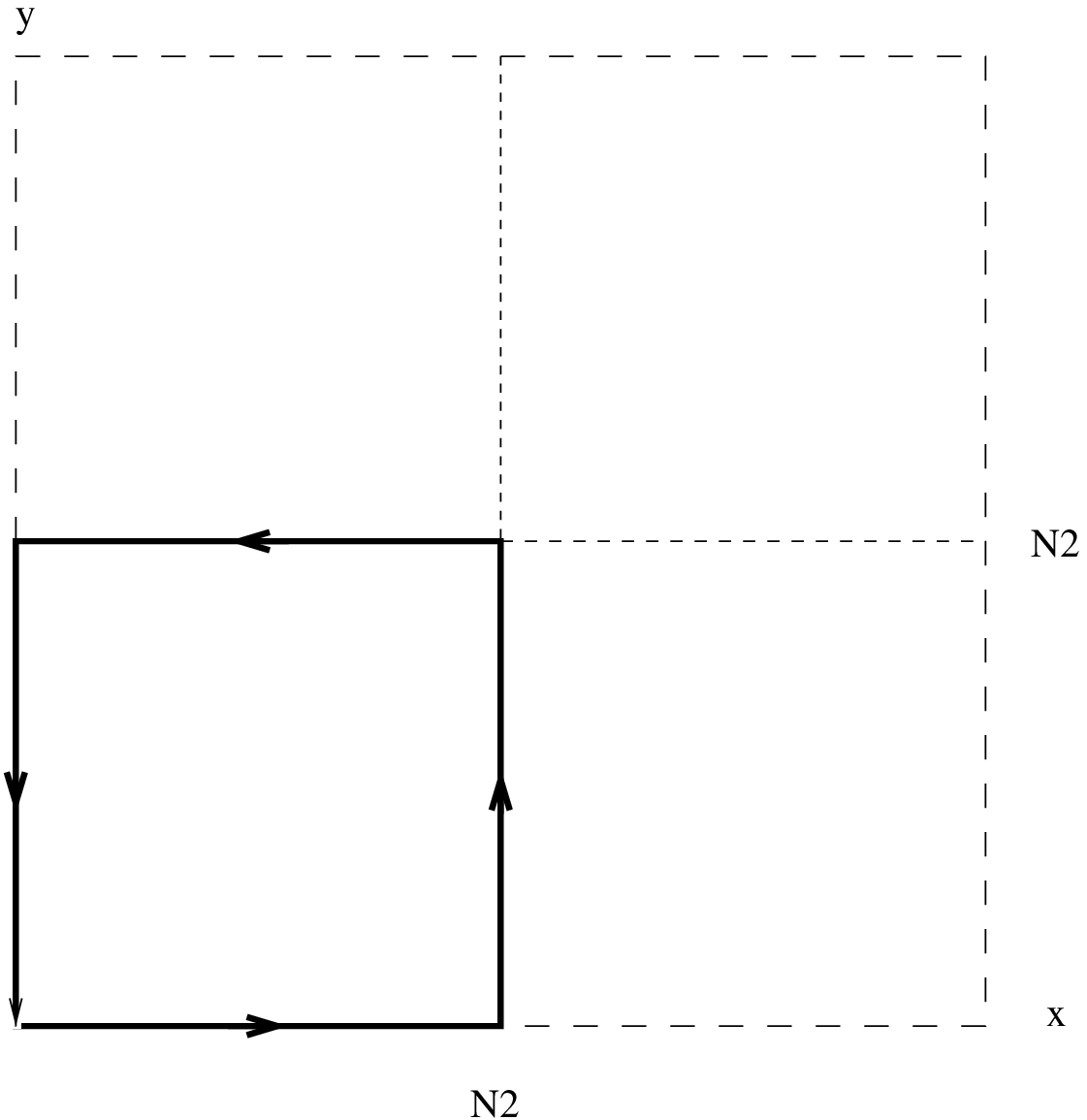} }   
\caption{Movement of lower left corner of subdomain $\Omega_4(t)$.}
\label{cornerfig}
\end{figure}

\begin{figure}[ht]
\psfrag{x}{$x$}
\psfrag{y}{$y$}
\vskip .5truein
  \centerline{\includegraphics[scale=.35 ]{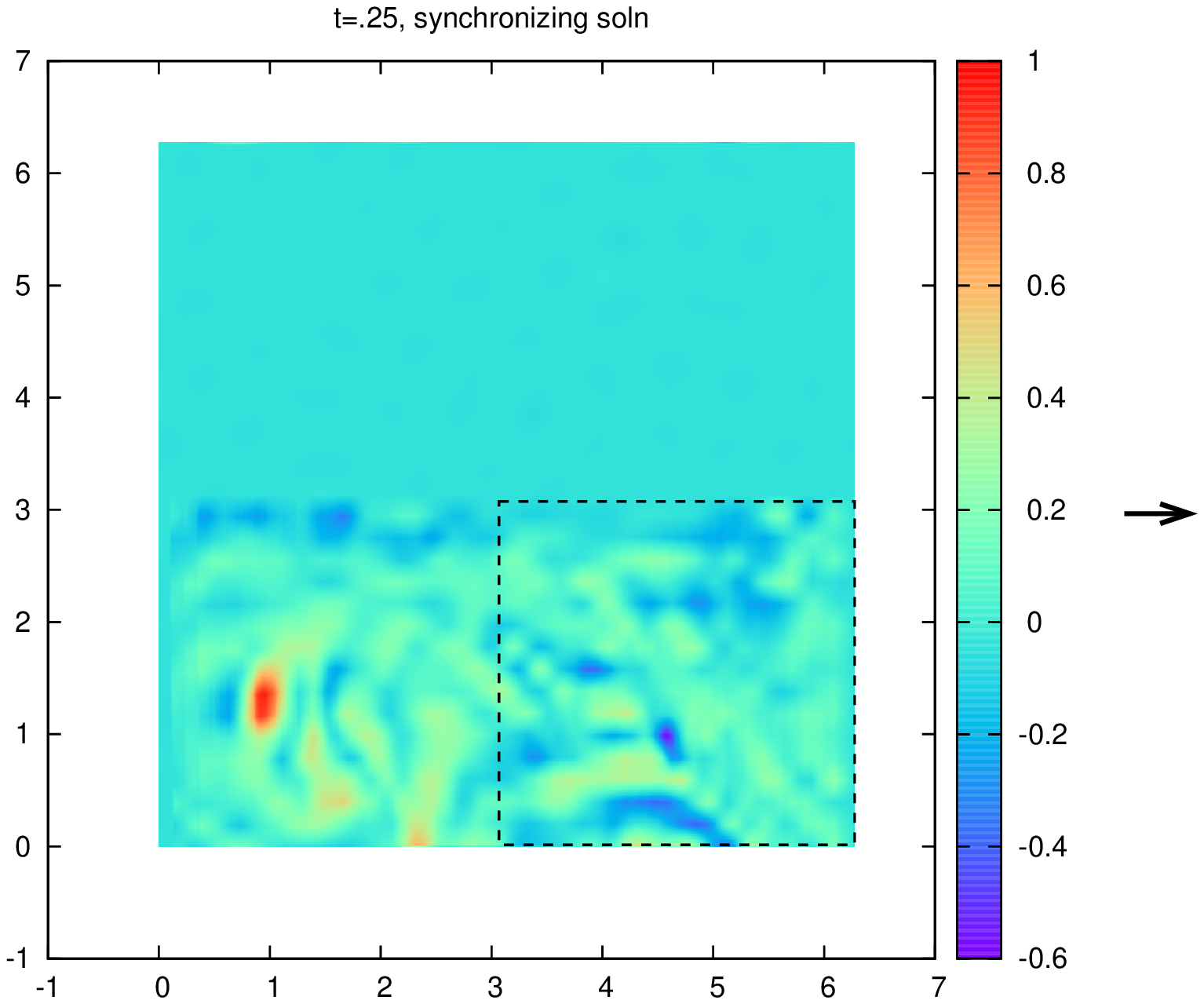}\hskip -.5 truein\includegraphics[scale=.35]{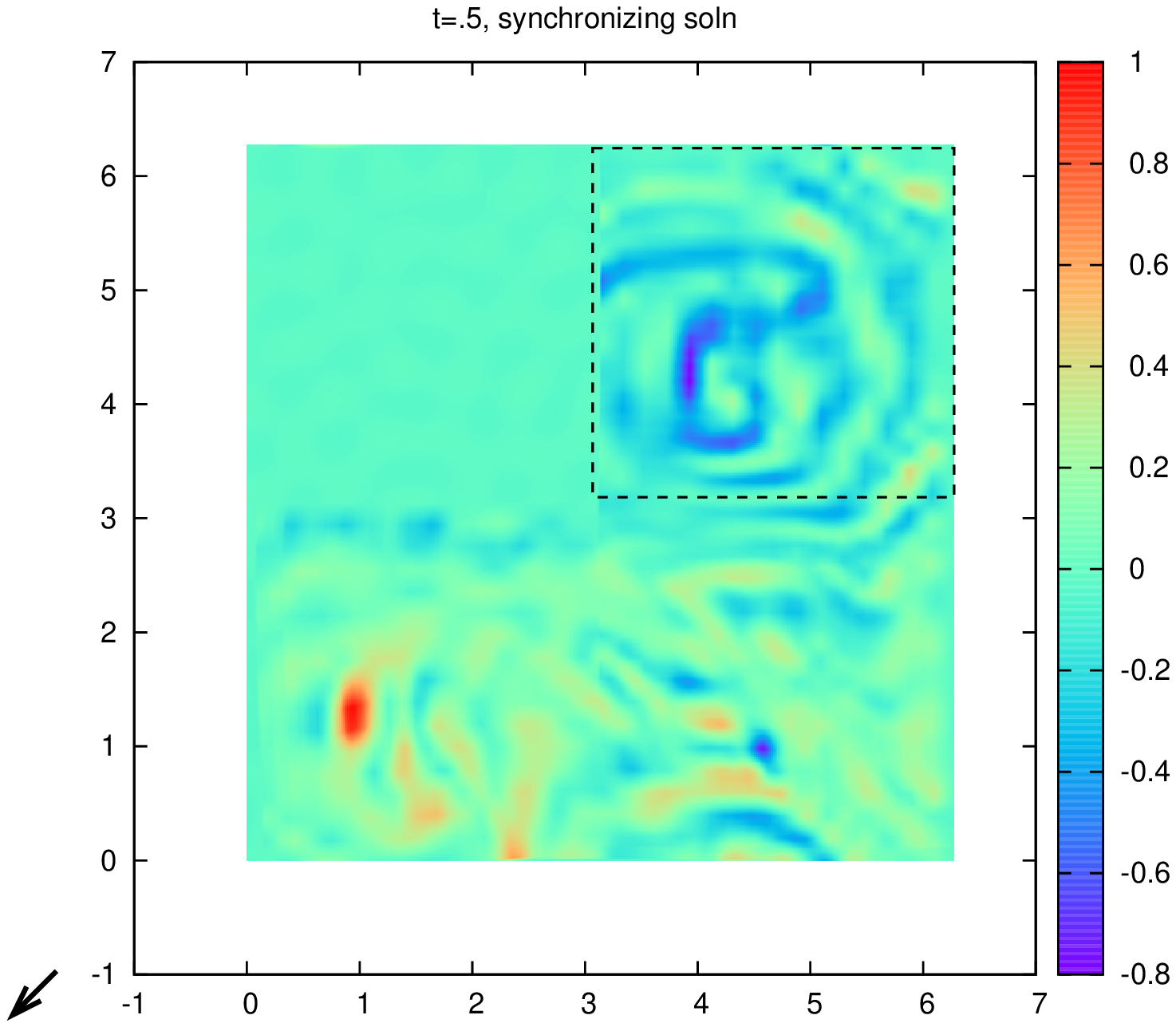} \hskip -.25 truein\includegraphics[scale=.35]{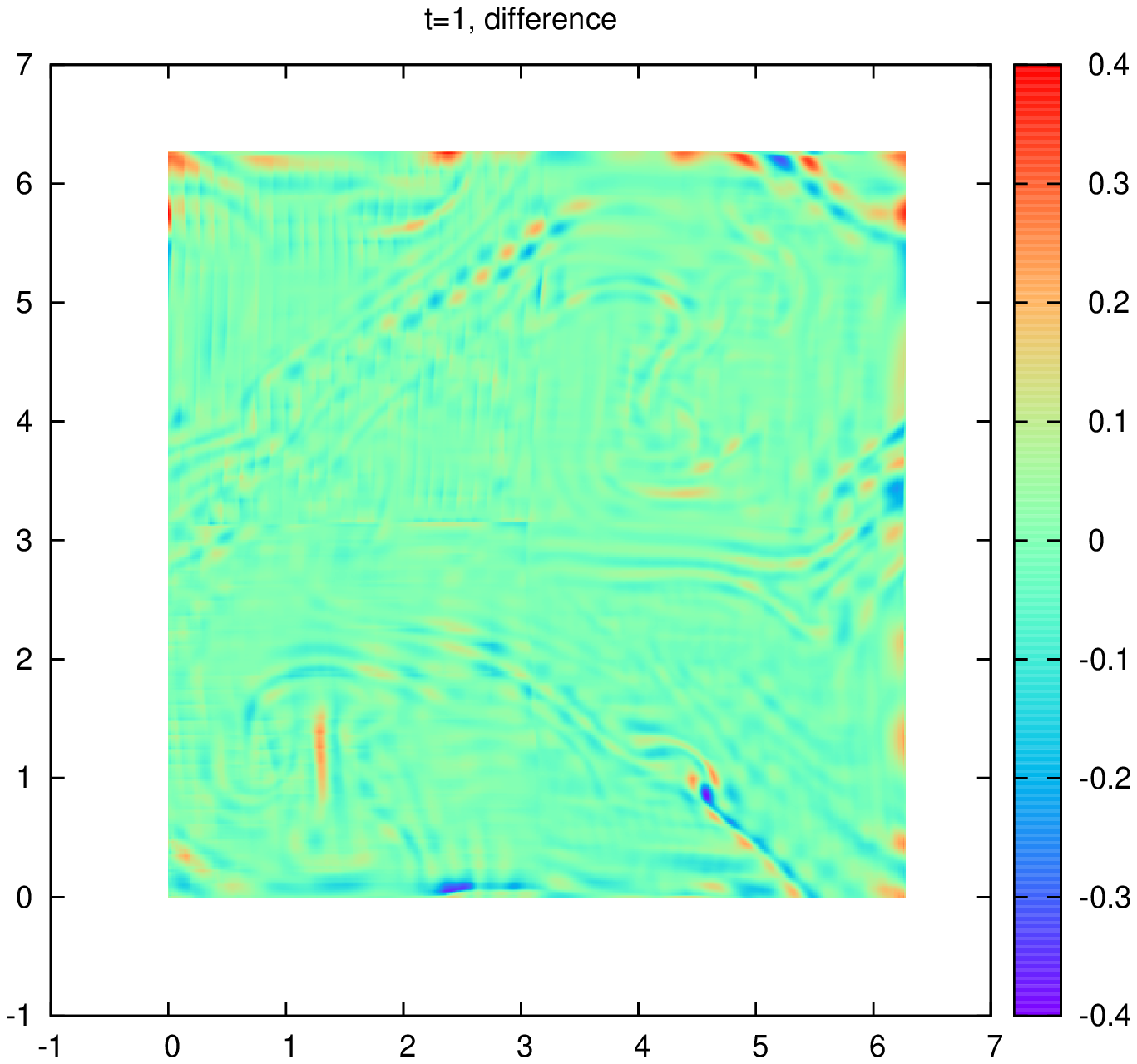}} \vskip .1truein
  \centerline{\includegraphics[scale=.35 ] {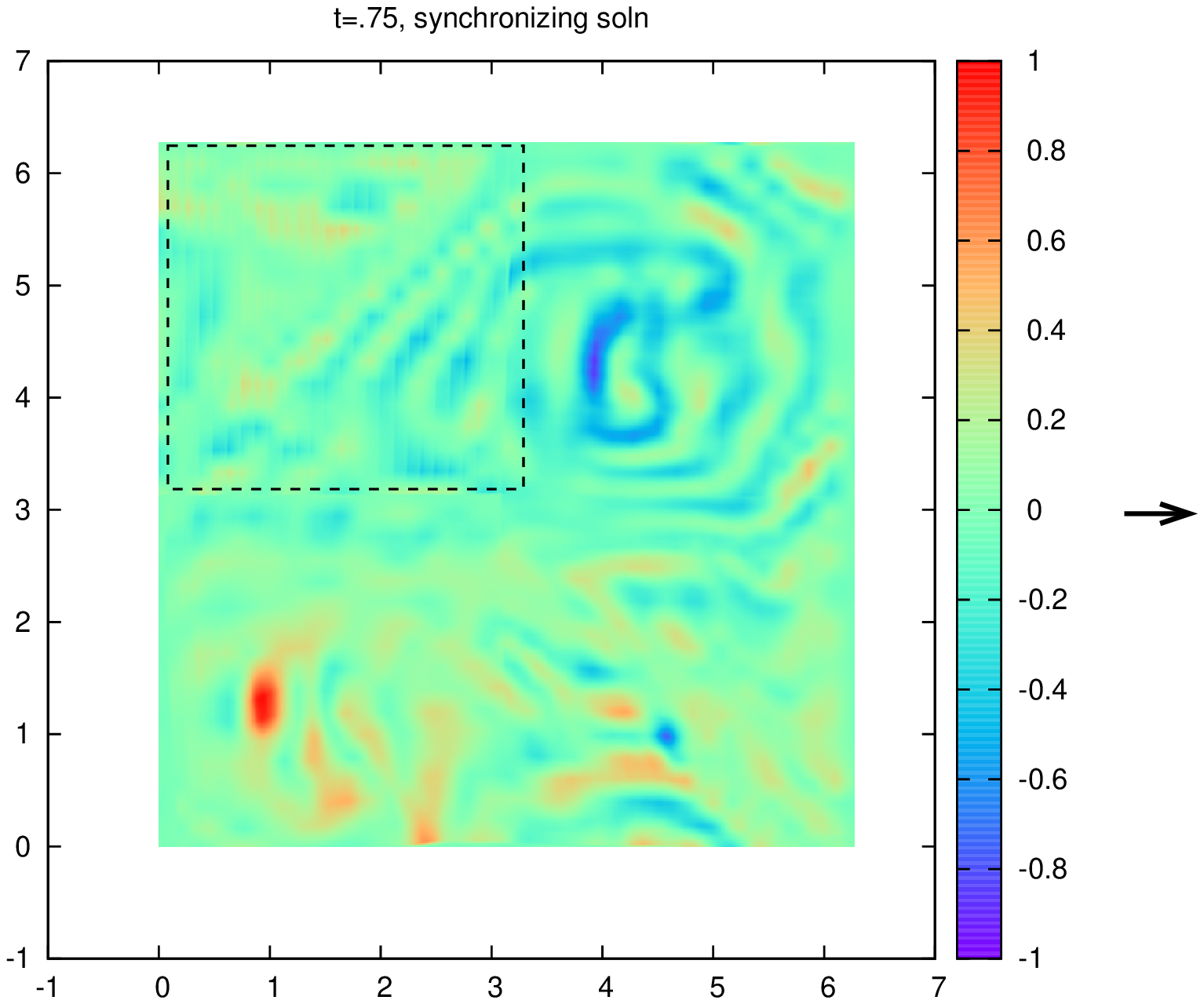} \hskip -.5 truein\includegraphics[scale=.35]{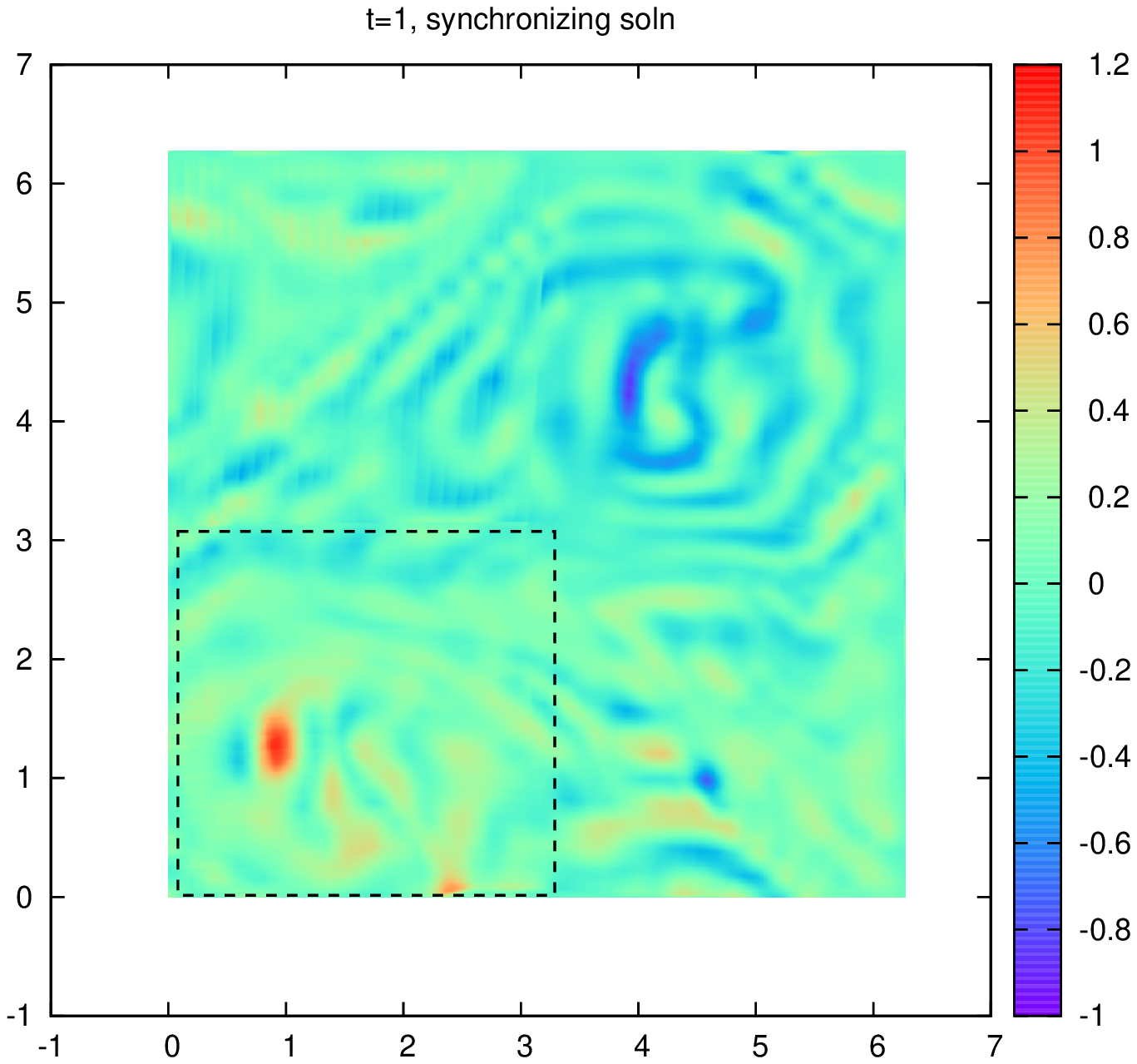}  \hskip -.25 truein\includegraphics[scale=.35]{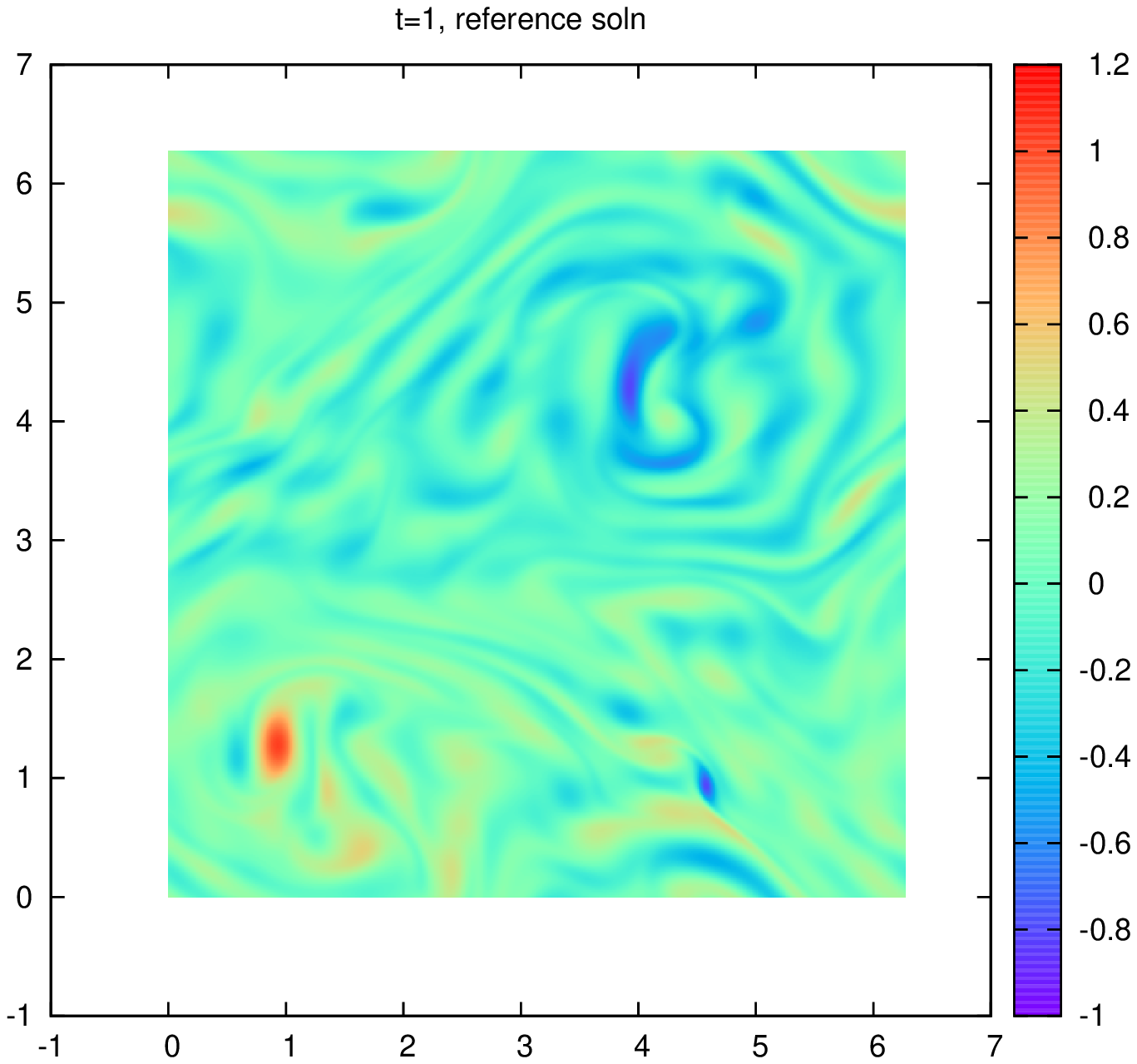}}
\caption{Initial cycle of nudging with $\Omega_4(t)$, $h=\pi/16$, starting at $t=.25$. Reference solution and difference at $t=1$ are on the right. }
\label{figuuu}
\end{figure}

A similar route can be taken by a subdomain $\Omega_5(t)$, where $|\Omega_5(t)|=1/16|\Om|$, such as that shown in Figure \ref{cornerfig2}. Note that in this case the periodic extension of $n_y(t)$ is discontinuous.  Though a bit slower than with $\Omega_4(t)$, synchronization is still achieved with this smallest small subdomain (see Figure \ref{errmobile}).

\begin{figure}[ht]
\psfrag{3}{\tiny$3$}
\psfrag{4}{\tiny$4$}
\psfrag{7}{\tiny$7$}
\psfrag{8}{\tiny$8$}
\psfrag{11}{\tiny$11$}
\psfrag{12}{\tiny$12$}
\psfrag{15}{\tiny$15$}
\psfrag{ix}{\tiny$n_x(t)$}
\psfrag{iy}{\tiny$n_y(t)$}
\psfrag{x}{\tiny$x$}
\psfrag{y}{\tiny$y$}
\psfrag{t}{\tiny$t$}
\psfrag{N2}{\tiny$N/2$}
\psfrag{N4}{\tiny$N/4$}
\psfrag{N3}{\tiny$3N/4$}
 \centerline{\includegraphics[scale=.5 ]{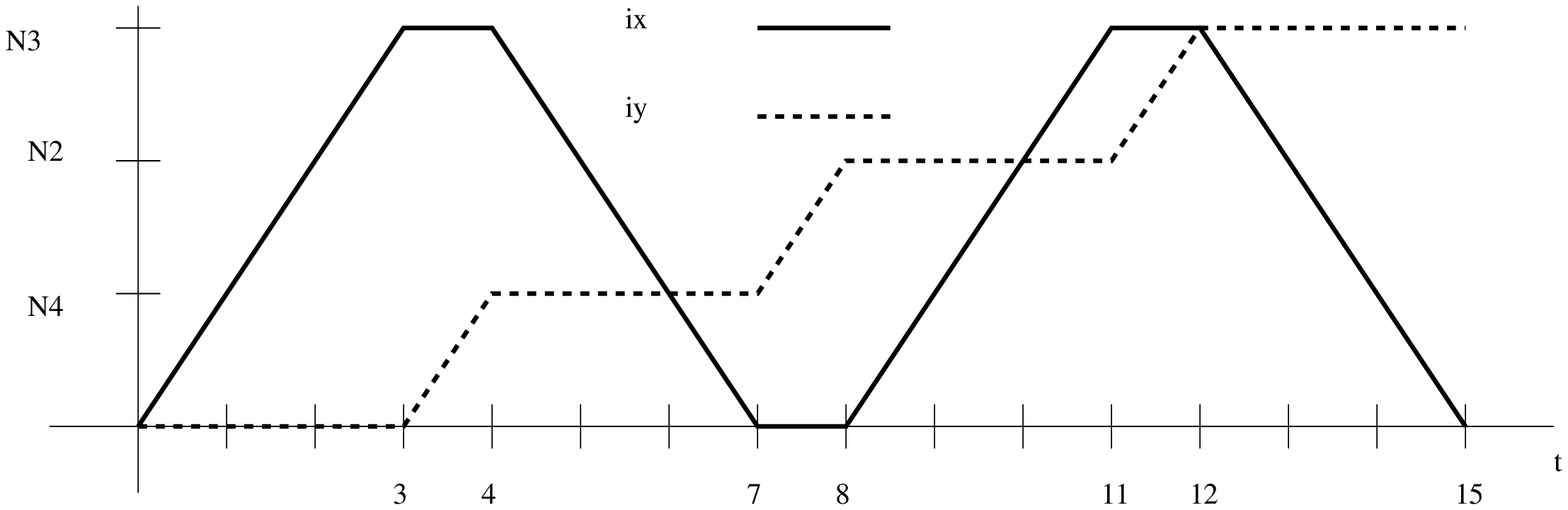}\quad \includegraphics[scale=.35 ]{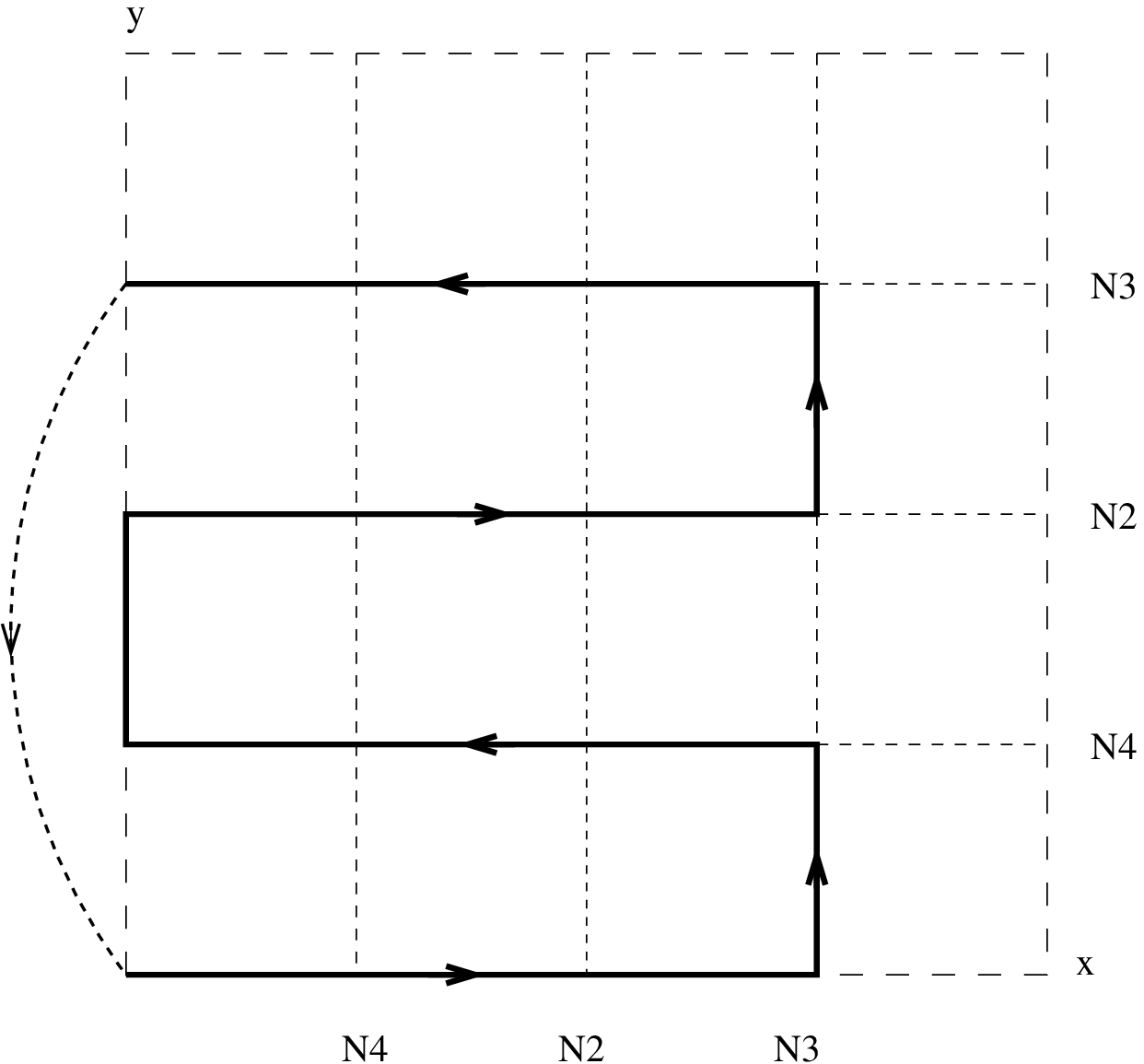} }
\caption{Movement of lower left corner of subdomain $\Omega_5(t)$.}
\label{cornerfig2}
\end{figure}

\begin{figure}[ht]
\psfrag{t}{\tiny$t$}
\psfrag{Omega}{\tiny$\Omega$}
\psfrag{omega}{\tiny$\omega$}
\psfrag{omegam}{\tiny$\omega^-$}
\psfrag{omegap}{\tiny$\omega^+$}
\psfrag{L2rel}{\tiny $L^2(\Om)$ rel. error}
\psfrag{'L2m127'}{\tiny$\Omega_5(t)$}
\psfrag{'L2m255'}{\tiny$\Omega_4(t)$}
\psfrag{'L2m127' using 1:3}{\tiny$\Omega_5(t)$}
\psfrag{'L2m255' using 1:3}{\tiny$\Omega_4(t)$}
\psfrag{Linf}{\tiny $L^\infty$ rel. error}
  \centerline{\includegraphics[scale=.35, angle=-90]{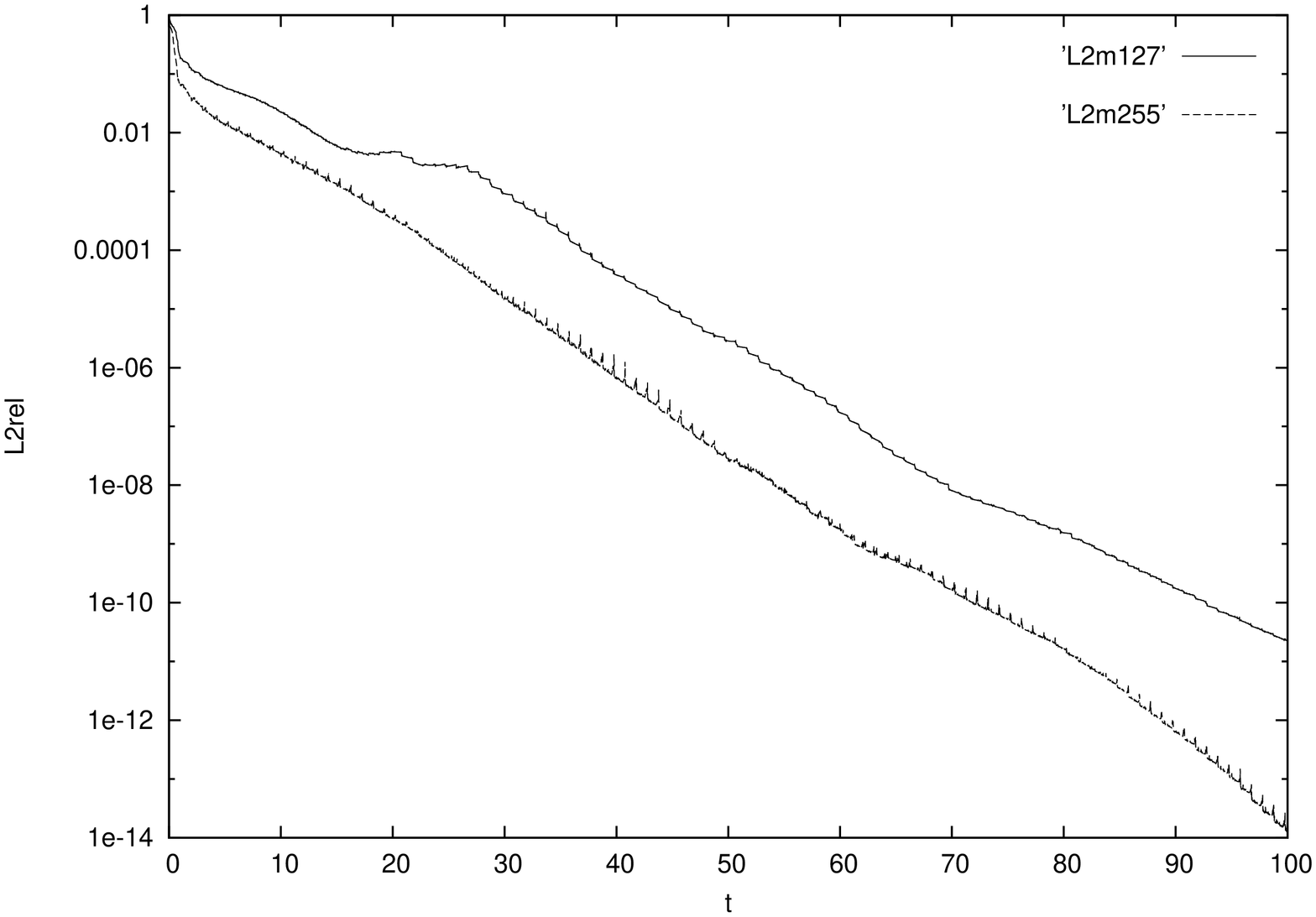} 
  \qquad \includegraphics[scale=.35, angle=-90]{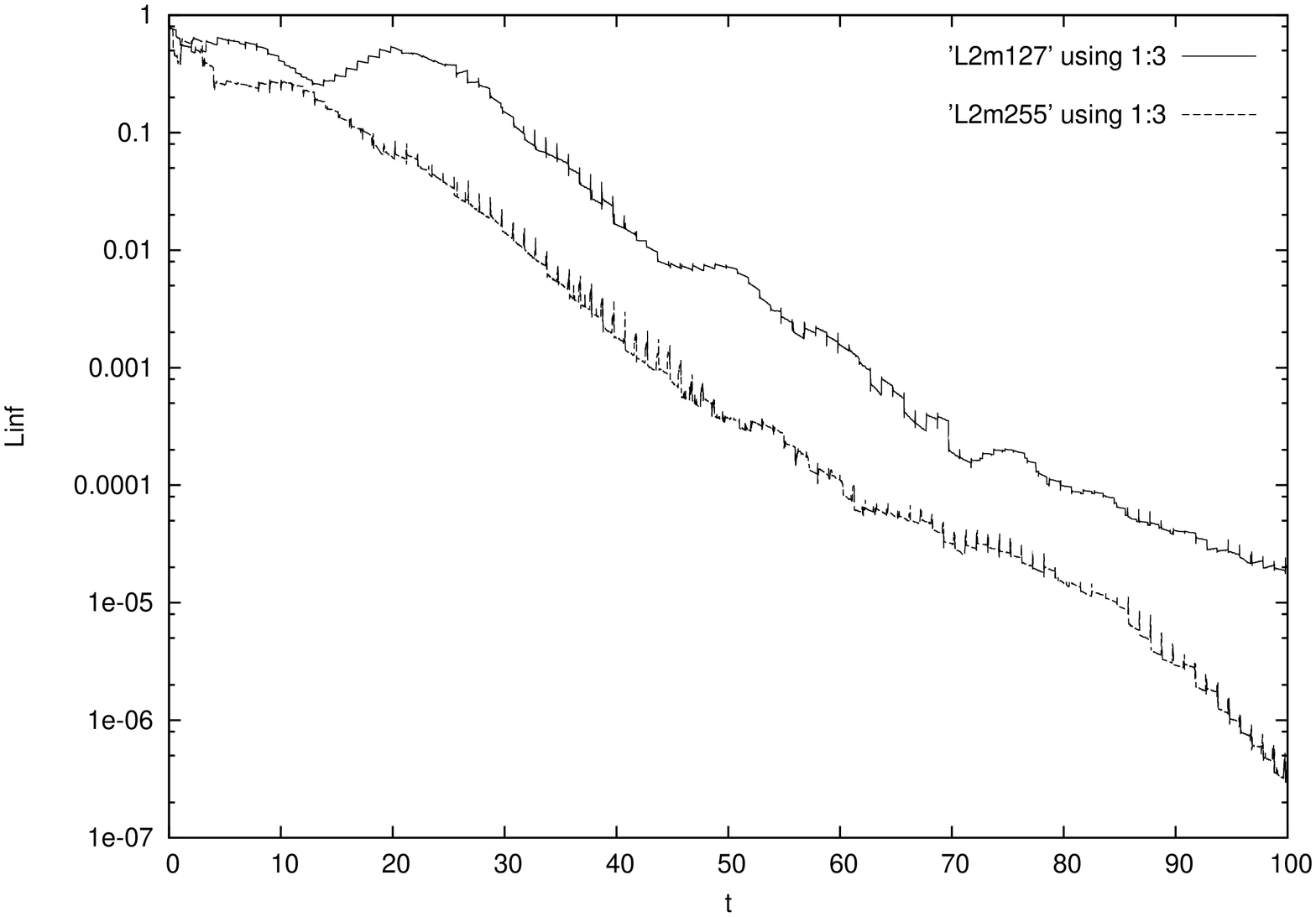}}
\vskip .15 truein
\caption{Relative errors for $\Omega_4(t)$, $\Omega_5(t)$, both with $h=\pi/16$. }
\label{errmobile}
\end{figure}


\section{Summary}\label{sec.summary}

Previous rigorous results on data assimilation in the direction of \cite{AOT} rely on uniformly distributed observations of the reference solution. We have rigorously shown that, modulo an arbitrarily small error, the observations can be confined to a sub-domain provided the solution is sufficiently regular and sufficiently many local samples are used.

Analysis guarantees what should work in practice (up to numerical error).   Conversely, when an algorithm works in practice, it suggests there might be some analysis to support it.   Computational work has demonstrated that nudging over the entire computational domain works much better than required in the rigorous estimates \cite{ATGKZ,OT1,FJJT,HJ,LRZ,LV}.  The conditions in Theorem \ref{thrm.main1} are essentially
$$
\mu \gtrsim \nu G^2 e^{\sqrt{N}}\;, \qquad h \lesssim \sqrt{\frac{\nu}{\mu}} \sim \frac{1}{G} e^{\sqrt N/2}.
$$
In our pseudospectral implementation, we have $h=2\pi/N$, so strictly speaking Theorem \ref{thrm.main1}  would require $N\sim G\exp(\sqrt N/2)$, which is far from obtainable.  

Yet our computational results are promising.  We have synchronized with a chaotic reference solution to near machine double precision in relative $L^2$ error using data on every 8th grid point (in each direction) on a fixed subdomain that is roughly $3/4 \times$ the area of the computational domain $\Omega_0$.  The rate of exponential decay in the error slows somewhat when data is restricted to a subdomain that is roughly $2/3 \times$ the area of $\Omega_0$.  The $L^2$ error does not appreciatively decay at all if data is taken on a subdomain of roughly $1/4 \times$ the area.  Still, the main features of the vorticity field are captured if data is taken on even just a quarter of the area.  Overall then, this constitutes another case of an algorithm working better than analysis suggests.

Preliminary tests of nudging on moving subdomains are even more encouraging.  Sliding subdomains of $1/4\times$ and even $1/16\times$ the area of $\Om$ to cover $\Om$  achieves synchronization in one-tenth the time needed for a larger fixed domain, and does so with coarser data.  This suggests analysis of mobile local data assimilation is merited, a matter we will explore in a future work.

\end{document}